\documentclass[preprint,3p,times]{elsarticle}

\usepackage{amssymb}
\usepackage{amsmath}
\usepackage{amsthm}
\usepackage{color}
\usepackage{graphicx,subfig} 
\usepackage{booktabs}   
\usepackage{threeparttable} 
\usepackage{enumerate}
\usepackage[all,pdf]{xy}
\usepackage{lmodern,amsmath}
\usepackage{tikz-cd}
\usepackage{mathrsfs}
\usepackage{lipsum}
\usepackage{enumitem,verbatim}
\usepackage{tikz}
\usepackage{appendix}
\usepackage[colorlinks=true]{hyperref}

\theoremstyle{plain}
  \newtheorem{thm}{Theorem}[section]
  \newtheorem{lem}[thm]{Lemma}
  \newtheorem{prop}[thm]{Proposition}
  \newtheorem{cor}[thm]{Corollary}
\theoremstyle{definition}
  \newtheorem{definition}[thm]{Definition}
  \newtheorem{exmp}[thm]{Example}
  \newtheorem{remark}[thm]{Remark}

\DeclareMathAlphabet{\mathcal}{OMS}{cmsy}{m}{n}

\makeatletter
\def\ps@pprintTitle{%
 \let\@oddhead\@empty
 \let\@evenhead\@empty
 \def\@oddfoot{\centerline{\thepage}}%
 \let\@evenfoot\@oddfoot}
\makeatother

\newcommand{\twoheaddownarrow}{\mathrel{\rotatebox[origin=c]{-90}{$\twoheadrightarrow$}}}
\renewcommand{\phi}{\varphi}

\numberwithin{equation}{section}

\allowdisplaybreaks

\begin{document}

\begin{frontmatter}



\title{A Natural Fuzzy Order on  Fuzzy Numbers}


\author{Hui Kou}
\ead{kouhui@scu.edu.cn}

\author{Mingchun Xie\corref{cor}}
\ead{xiemc1011@163.com}

\cortext[cor]{Corresponding author.}
\address{School of Mathematics, Sichuan University, Chengdu 610064, China}

\begin{abstract}
This paper introduces a natural fuzzy order on fuzzy numbers that extends the natural orders on real numbers and interval numbers. We investigate its completeness properties and show that the space of uniformly bounded fuzzy numbers is conically complete and conically cocomplete, and that it is complete if and only if the underlying continuous t-norm is the Gödel t-norm. Moreover, it is proved that this space constitutes a \([0,1]\)-enriched domain if and only if the underlying continuous t-norm satisfies the (S) condition. These results provide a foundation for ordering fuzzy numbers.
\end{abstract}

\begin{keyword}
Fuzzy numbers, Fuzzy order, [0,1]-enriched categories, Completeness, Continuity.  

\end{keyword}

\end{frontmatter}

\section{Introduction}
The theory of fuzzy sets, pioneered by Zadeh \cite{ZADEH1965338} in 1965, introduced a powerful mathematical framework for dealing with vagueness and uncertainty. Its fundamental innovation lies in the concept of membership functions, which extend the notion of an element's belonging to a set from a binary truth value to a continuous degree within the interval \([0,1]\). This allows for the precise mathematical characterization of imprecise concepts. Since its inception, fuzzy mathematics has evolved rapidly, developing a rich theoretical structure encompassing fuzzy logic, fuzzy relations, and fuzzy systems. Its applications span diverse fields such as pattern recognition, automatic control, and decision analysis, profoundly transforming methodologies for modeling and solving problems involving uncertainty. As a core object within this theory, fuzzy numbers provide a natural model for imprecise quantitative information, and the problem of comparing and ordering them has consequently become a foundational topic.

Within this domain, the comparison and ordering of fuzzy numbers represents a fundamental and persistent theoretical challenge. The effectiveness of ranking fuzzy numbers directly influences the reliability of applications such as fuzzy optimization and multi-attribute decision-making.

Previously, many scholars attempted to address this challenge by relying on classical crisp partial orders. For instance, Zadeh \cite{ZADEH1971177} defined an order based on the pointwise comparison of membership functions, while Klir and Yuan \cite{Klir1995FuzzySA} defined one based on comparisons of all $\alpha$-cuts. Subsequent research has further explored properties of these orders, such as the existence of supremum and infimum for bounded sets and order approximation under specific distance metrics. However, these relations are all partial orders. While they are theoretically important, their inherent allowance for incomparability makes them difficult to apply directly in practical decision scenarios that require a complete ranking. This gap between theoretical soundness and practical need for a total order has motivated the development of various ranking methodologies, which can be broadly categorized into three types.

The first category comprises defuzzification and feature extraction techniques, which map a fuzzy number to a single real value to enable direct comparison. Common approaches include calculating the centroid, the center of area, or specific points from the support or core of the fuzzy set \cite{BORTOLAN19851,KaufmannGupta1985,CHU2002111}. Methods based on distance metrics---such as minimizing the distance between fuzzy numbers \cite{CHENG1998307}---also belong to this category. Although these techniques induce a linear order and thus simplify comparison, they often discard essential fuzzy information, such as the spread of the support or the shape of the membership function.

The second category involves constructing a reference set \cite{jain1977procedure,jain1976decision,CHEN1985113} (e.g., an ideal or anti-ideal fuzzy number). Fuzzy numbers are then ranked based on their distance or similarity to this reference. A notable drawback of this approach is the inherent subjectivity involved in defining the reference set.

The third category involves constructing a fuzzy relation to perform pairwise comparisons between fuzzy numbers \cite{wang1997comparative,wang1996classification}. A final ranking is then derived by analyzing this relational matrix. Wang and Kerre \cite{WANG2001387} proved that many fuzzy relations used for comparing fuzzy quantities satisfy conditions stronger than acyclicity, thereby providing a widely applicable framework for deriving a total order from a fuzzy relation. A fuzzy order is formally defined as a generalization of a crisp order to the fuzzy setting. Notably, however, none of the specific fuzzy relations discussed by Wang and Kerre \cite{WANG2001387} constitute a fuzzy order. More recently, Urbański \cite{URBANSKI2024108992} studied fuzzy orders for the level sets of fuzzy numbers based on triangular norms. Beyond these,  other approaches include the parameter-dependent method in \cite{WANG2014131}, which achieves a linear order but relies on parameter selection and results in a lexicographic (not natural) order for interval-valued fuzzy numbers. Alternatively, methods based on probability measures \cite{ZADEH19999,DUBOIS1983183} frequently lack properties such as reflexivity and weak transitivity, preventing them from forming a true fuzzy order.

In summary, existing ranking methods face a fundamental trade-off. They either sacrifice the strict mathematical structure of an order to achieve a total order (e.g., defuzzification, reference-set methods), or they employ fuzzy relations that lack the support of complete order axioms. What remains missing is a mathematical construct that serves as a sound generalization of the classical partial order in the fuzzy setting—one that satisfies properties such as fuzzy reflexivity and transitivity. Although the abstract notion of a fuzzy order has been defined, constructing a concrete, natural instance for fuzzy numbers remains an open problem. Such an instance must be both an intuitive extension of the real order and endowed with rich theoretical content.

This paper aims to bridge this gap by defining a natural fuzzy order \(P\) for fuzzy numbers.  This order is designed not only as a direct and intuitive extension of the natural orders on real numbers \(\mathbb{R}\) and interval numbers \(\mathbb{IR}\), but also as one that incorporates the distribution and structure of fuzzy numbers. Our work is based on the theory of $[0,1]$-enriched categories---a framework in which such a category is essentially a set equipped with a fuzzy order. Leveraging recent advances in this theory~\cite{Zhang2024IntroductoryNO}, the fuzzy order relation \(P\) is formally defined and its fundamental order-theoretic properties are systematically investigated. \textbf{The main contributions of this work are as follows:}
\begin{itemize}
	\item \textbf{Definition of a novel fuzzy order:} Based on \([0,1]\)-enriched category theory, we propose a fuzzy order \(P\) for fuzzy numbers. This relation naturally extends both the natural order on real numbers and that on interval numbers. Moreover, it allows the degree of comparison to be modulated by the choice of different continuous t-norms, while also reflecting the distribution and structure of fuzzy numbers.
	\item \textbf{Completeness analysis:} As a generalization of real numbers, fuzzy numbers inherit mathematical structures under this fuzzy order that are analogous to the completeness and bounded completeness of the real numbers. In particular, the uniformly bounded space is conically complete and conically cocomplete, and it is complete (equivalently cocomplete, tensored, or cotensored) exactly for the Gödel t-norm.

	\item \textbf{Continuity and domain structure:} Under t-norms satisfying condition (S), we establish the continuity of uniformly bounded fuzzy numbers and prove that they also have approximation properties similar to those on the real numbers.
	\item \textbf{Theoretical framework:} By providing a unified theoretical foundation, our work not only supports comparison and ranking in fuzzy decision-making but also opens avenues for further analytical developments in fuzzy mathematics.
\end{itemize}

The remainder of this paper is structured as follows. Section 2 provides a review of preliminary concepts related to fuzzy numbers and $[0,1]$-enriched categories. Section 3 introduces the decomposition of fuzzy numbers and defines the fuzzy order $P$. Section 4 shows various completeness properties of \([0,1]\)-enriched categories in the context of fuzzy numbers. Section 5 investigates the continuity and domain structure under the condition (S).
\section{Preliminaries} 
This section reviews the fundamental concepts for constructing the fuzzy order. The section begins with the standard definition of fuzzy numbers. Then triangular norms (t-norms) and their residual implications are introduced, which are the primary logical operators for defining the fuzzy order. Furthermore, we recall key notions from the theory of $[0,1]$-enriched categories. 
\begin{definition}
The fuzzy number space \( \mathbb{E}^1 \) is the set of all functions \( u :\mathbb{R}\to [0,1]\) satisfying the following properties:
\begin{enumerate}[label={\rm(\arabic*)}]
  \item Normality: There exists \( x_0 \in \mathbb{R} \) with \( u(x_0) = 1 \).
  \item Convexity: \( u(\lambda x + (1-\lambda)y) \geq \min \{u(x), u(y)\} \) for all \( x, y \in \mathbb{R}, \lambda \in [0,1] \).
  \item Upper-semicontinuity: \( u(x) \) is upper-semicontinuous, i.e.,  for all  \(x_0\in \mathbb{R}, \limsup_{x \to x_0} u(x) \le u(x_0).\) 
  \item Compact support: \([u]_0=\operatorname{cl}_{\mathbb{R}} \left\{x \in \mathbb{R} : u(x) > 0\right\} \) is compact in \( \mathbb{R} \).
\end{enumerate}
\end{definition} 
Given a fuzzy subset \( u \) on \( \mathbb{R} \), the \( \alpha \)-level set of \( u \) is defined by \([u]_ \alpha  = \{x \in \mathbb{R} : u(x) \geq  \alpha \} \) for \(  \alpha  \in (0, 1] \).

\begin{definition}[\cite{Alsina2006}]
    A \( t \)-norm is a function of two variables \( \&: [0,1] \times [0,1] \to [0,1] \) that satisfies the following properties:
\begin{enumerate}[label={\rm(\arabic*)}]
    \item Commutativity: for every \( x, y \in [0,1] \) one has \( x \&y = y\&x \),
    \item Monotonicity: for every \( x, y, z, w \in [0,1] \) if \( x \leq z \), \( y \leq w \) then  $x \& y \leq z\& w$,
    \item Associativity: for every \( x, y, z \in [0,1] \) one has \( x\& (y\& z) = (x\& y)\& z \),
    \item Identity element: for every \( x \in [0,1] \) one has \( x\& 1 = x \).
\end{enumerate}
\end{definition}
We denote by $\oplus$ the standard fuzzy addition defined via Zadeh's extension principle
(cf.\ \cite{Klir1995FuzzySA}), which satisfies $[u\oplus w]_\alpha=[u]_\alpha+[w]_\alpha$ for all
$\alpha\in[0,1]$.

Given a left-continuous t-norm \(\& \), the symbol \( \rightarrow \) denotes its unique residual implication. 
\[
x\rightarrow y = \sup\{u \in [0,1] \mid x\& u \leq y\}.
\]
We briefly review the key properties of the residual implication, as discussed in \cite{HL2006}.
\begin{prop}\label{properties of residual implication}
    Suppose that \( \& \) is a left continuous t-norm and \( \rightarrow \) is the residual implication with respect to \( \& \). Then
\begin{enumerate}[label={\rm(\arabic*)}]
    \item \( x \leq y \) if and only if \( x\rightarrow y = 1 \).
    \item \( x\& y \leq z \) if and only if \( x \leq y\rightarrow z \).
    \item \( (x\rightarrow y)\& (y\rightarrow z) \leq x\rightarrow z \).
    \item \( 1\rightarrow y = y \).
    \item \( x\&( x \rightarrow y) \leq y \).
    \item \(( \bigvee_{i} x_i)\rightarrow y=\bigwedge_i(x_i\rightarrow y)\).
    \item \( x\rightarrow(\bigwedge_i y_i) =\bigwedge_i (x\rightarrow y_i)\).
\end{enumerate}
\end{prop}
\begin{prop}\label{left-continuous is jointly upper semicontinuous function}
	Let \(\&\) be a left-continuous t-norm and let \(\rightarrow\) be the residual implication with respect to \(\&\). Then, \(I(a, b) = a\rightarrow b \) is a jointly upper-semicontinuous
function in \([0,1]^2\).
\end{prop}
\begin{proof}
For every \(c\in[0,1]\), define the \(c\)-upper-level set
\[
L_c:=\{(a,b)\mid I(a,b)\ge c\}.
\]
We prove that each \(L_c\) is closed. Fix \(c\in[0,1]\) and take a sequence \((a_n,b_n)\in L_c\) with \((a_n,b_n)\to(a,b)\).  
 
\noindent\textbf{Case 1.} There are infinitely many indices \(n\) such that \(a_n\le a\).  
Then a subsequence \(\{(a_{n_k},b_{n_k})\}_{n_k}\) can be chosen  with \(a_{n_k}\le a\) and \(a_{n_k}\nearrow a\).  
Since \(\&\) is left‑continuous in its first variable and \(a_{n_k}\le a\), it follows that
\[
 a\,\&\,c=\lim_{k\to\infty} (a_{n_k}\,\&\,c) \le \lim_{k\to\infty} b_{n_k}=b .
\]
\noindent\textbf{Case 2.} Only finitely many indices satisfy \(a_n\le a\).  
Then there exists \(N\) such that for all \(n>N\), \(a_n > a\).  
Monotonicity of the t‑norm gives \(a\,\&\,c \le a_n\,\&\,c\) whenever \(a_n>a\). Consequently, for \(n>N\),
\[
a\,\&\,c \le a_n\,\&\,c \le b_n .
\]
Letting \(n\to\infty\) yields  \(a\,\&\,c \le b\).
Thus, in both cases \(a\,\&\,c\le b\), so \((a,b)\in L_c\). Hence each \(L_c\) is closed, and the residuum \(I\) is jointly upper‑semicontinuous on \([0,1]^2\).  
\end{proof}
\begin{definition}[\cite{Gierz}]
   Suppose \(P\) is a partially  ordered set. We say that 
 \begin{enumerate}[label={\rm(\arabic*)}]
     \item \(P\) is a directed complete partially ordered set, a dcpo for short, if every directed subset of $P$ has a join. \(P\) is a complete lattice  if  any subset \(A\) of \(P\) has a join.
     \item \(P\) is a continuous partially ordered set if for each $x \in P$, the set $\{y \in P \mid y \ll x\}$ is directed and has $x$ as a join, where  \(y \ll x \), if for each directed \(D\) of \(P\) with a join, \[x\le \sup D \implies y\le d \text{ for some \(d\in D.\)}\]
     \item \(P\) is a continuous lattice if \(P\) is a continuous and complete lattice.
 \end{enumerate}

\end{definition}
Next, some basic concepts of \([0,1]\)-enriched categories are introduced, and  notations are established. The materials can be found in \cite{Zhang2024IntroductoryNO,LAI20201,Lai2019CompletelyDE}. 
\begin{definition}
	A \([0,1]\)-enriched category is a pair $(X, \alpha)$, where $X$ is a set and $\alpha: X \times X \to [0,1]$ is a function such that
	\begin{enumerate}[label={\rm(\arabic*)}]
		\item $\alpha(x, x) = 1$ for all $x \in X$;
		\item  $\alpha(y, z) \& \alpha(x, y) \leq \alpha(x, z)$ for all $x, y, z \in X$.
	\end{enumerate}
\end{definition}
For a $[0,1]$-enriched category $(X,\alpha)$, its opposite category
$X^{\mathrm{op}}$ is defined by
\(
X^{\mathrm{op}}(x,y)=X(y,x).
\)
The underlying relation of $X$ is defined by
\(
x\sqsubseteq y
\Longleftrightarrow
X(x,y)=1.
\)
We write
\(
X_0=(X,\sqsubseteq).
\)
In general, $X_0$ is a preorder. If $X$ is separated, then
$X_0$ is a partially ordered set.
For notational convenience, when dealing with a \([0,1]\)-enriched category \((X,\alpha)\), the symbol \(\alpha\) is often omitted and we write $X(x,y)$ for $\alpha(x,y)$.
Two elements $x$ and $y$ of a \([0,1]\)-enriched category $X$ are isomorphic if $X(x,y) = 1 = X(y,x)$. A \([0,1]\)-enriched category $X$ is separated if its isomorphic elements are identical.
\begin{definition}[Fuzzy preorder and fuzzy order \cite{HL2006}]\label{def:fuzzy-order}
Let $X$ be a set and $\&$ a left-continuous t-norm. A map $P \colon X \times X \to [0,1]$ is called
\begin{enumerate}[label=(\arabic*)]
    \item an \emph{$\&$-fuzzy preorder} on $X$ if
        \begin{itemize}
            \item[(i)] $P(x,x) = 1$ for all $x \in X$ \hfill (reflexivity);
            \item[(ii)] $P(y,z) \,\&\, P(x,y) \leq P(x,z)$ for all $x,y,z \in X$ \hfill (transitivity).
        \end{itemize}
    \item a \emph{fuzzy order} (or \emph{$\&$-fuzzy partial order}) on $X$ if it is an $\&$-fuzzy preorder 
        and additionally satisfies
        \begin{itemize}
            \item[(iii)] $P(x,y) = P(y,x) = 1$ implies $x = y$ \hfill (antisymmetry).
        \end{itemize}
\end{enumerate}
The pair $(X,P)$ is called a \emph{fuzzy preordered set} (resp.\ \emph{fuzzy ordered set}).
\end{definition}
\begin{remark}
Let \(X\) be a \([0,1]\)-enriched category. If we interpret the value \(\alpha(x,y)\) as the truth degree that \(x\) is less than or equal to \(y\), then (1) corresponds to reflexivity and (2) to transitivity. If we replace \([0,1]\) with \(\{0,1\}\), then \(X\) becomes a preorder. If \(X\) is separated, then it becomes a partially ordered set. Therefore, a separated \([0,1]\)-enriched category is exactly a fuzzy ordered set in the sense of Definition~\ref{def:fuzzy-order}.
\end{remark}
\begin{exmp}
\begin{enumerate} [label={\rm(\arabic*)}]
    \item The pair $([0,1], \alpha_L)$ is a separated \([0,1]\)-enriched category, where for all $x, y \in [0,1]$,
$\alpha_L(x, y) = x \to y$. The opposite of $([0,1], \alpha_L)$ is denoted by $([0,1], \alpha_R)$; that is, $\alpha_R(x, y) = y \to x$. In the sequel we write $V$ for $([0,1], \alpha_L)$, hence $V^\mathrm{op}$ for $([0,1], \alpha_R)$. Both $V$ and $V^\mathrm{op}$ play an important
role in the theory of \([0,1]\)-enriched categories.
    \item For each set $X$, $([0,1]^X, \text{sub}_X)$ is a separated \([0,1]\)-enriched category, where for all $\lambda, \mu \in [0,1]^X$,
  \[
  \text{sub}_X(\lambda, \mu) = \inf_{x \in X} \lambda(x) \to \mu(x).
  \]
Following Zadeh \cite{ZADEH1965338}, when $\lambda$ and $\mu$ are interpreted as fuzzy subsets of $X$, the value $\text{sub}_X(\lambda, \mu)$ measures the truth degree of the inclusion $\lambda \subseteq \mu$. Consequently, the category $([0,1]^X, \text{sub}_X)$ is termed the enriched powerset of $X$.
\end{enumerate}
 
\end{exmp}
    
Suppose $X,Y$ are \([0,1]\)-enriched categories. A functor $f:X \to Y$ is a map such that $X(x,y) \leq Y(f(x),f(y))$ for all $x,y \in X$. A weight of \( X \) is defined to be a functor \( \phi: X^{\mathrm{op}} \to V \), where \( V = ([0,1], \rightarrow) \), i.e., \(X(y,x)\le \phi(x)\rightarrow \phi(y)\). Weights of $X$ constitute a \([0,1]\)-enriched category \( \mathcal{P}X\) with \[ \mathcal{P}X(\phi_1,\phi_2)=\operatorname{sub}_X(\phi_1,\phi_2)=\inf_{x\in X}\left(\phi_1(x)\rightarrow \phi_2(x)\right).\]
Dually, a coweight of $X$ is a functor $\psi: X \to V$, where $V= ([0,1], \rightarrow)$. All coweights of
$X$ constitute a \([0,1]\)-enriched category $\mathcal{P}^\dagger X$ with
\[
\mathcal{P}^\dagger X(\psi_1, \psi_2)= \operatorname{sub}_X(\psi_2, \psi_1)=\inf_{x\in X}\left(\psi_2(x)\rightarrow \psi_1(x)\right).
\]
\begin{definition}
    A colimit of a weight $\phi$ is an element $\operatorname{colim}\phi$ of $X$ such that for all $x \in X$,
\[
X(\operatorname{colim}\phi, x) = \mathcal{P}X(\phi, X(-,x)).
\]
A limit of a coweight $\psi$ is an element $\operatorname{lim}\psi$ of $X$ such that for all $x \in X$,
\[
X(x, \operatorname{lim}\psi ) = \mathcal{P}^\dagger X( X(x,-), \psi).
\]
\end{definition}
Clearly, each weight has, up to isomorphism, at most one colimit, and similarly each coweight has at most one limit.

\begin{definition}
	A \([0,1]\)-enriched category $X$ is cocomplete if every weight of $X$ has
	a colimit.
\end{definition}

Let $X$ be a \([0,1]\)-enriched category.  For each $x\in X$ and $r\in [0,1]$, the tensor of $r$ with $x$, denoted by $r\otimes x$, is an element of $X$ such that for all $y\in X$,$$X(r\otimes x,y)=r\rightarrow X(x,y).$$
Tensors are a special kind of colimits. 
For each $y\in X$ and $r\in [0,1]$, the cotensor of $r$ with $y$, denoted by $r\multimap  y$, is an element of $X$ such that for all $x\in X$,$$X( x,r\multimap y)=r\rightarrow X(x,y).$$
Cotensors are a special kind of limits.

A weight $\phi$ of $X$ is conical if \[
\phi = \sup_{a \in A} X(-,a)
\]
for a subset $A$ of $X$. \(\phi\) is a finite conical weight if  \(A\) is finite.  Dually, a coweight $\psi$ of
$X$ is conical if
\[
\psi = \sup_{b \in B} X(b,-)
\]
for a  subset $B$ of $X$.   \(\psi\) is a finite conical coweight if  \(B\) is finite.

\begin{definition}
	Suppose \( X \) is a \([0,1]\)-enriched category. We say that
	\begin{enumerate}[label={\rm(\arabic*)}] 
		\item \( X \) is conically cocomplete if every conical weight of \( X \) has a colimit.
		\item \( X \) is finitely  conically cocomplete if every  finitely conical weight of \( X \) has a colimit.
		\item \( X \) is conically complete if every conical coweight of \( X \) has a limit.
		\item \( X \) is finitely  conically complete if every  finitely conical coweight of \( X \) has a limit.
		\item \( X \) is order-complete if the underlying ordered set \( X_0 \) is complete.
	\end{enumerate}
\end{definition}
\begin{thm}[\cite{kelly1982basic,Stubbe}, or \cite{Zhang2024IntroductoryNO}, Theorem 7.8]\label{complete category} 
	For each \([0,1]\)-enriched category $X$, the
	following statements are equivalent:
	\begin{enumerate}[label={\rm(\arabic*)}]
		\item $X$ is cocomplete.
		\item $X$ is complete.
		\item $X$ is tensored and conically cocomplete.
		\item $X$ is cotensored and conically complete.
		\item $X$ is order-complete, tensored and cotensored.
	\end{enumerate}
\end{thm}
\begin{definition}
	Suppose $\{x_i\}_{i \in D}$ is a net and $b$ is an element of a \([0,1]\)-enriched
	category $X$. We say that
	\begin{enumerate} [label={\rm(\arabic*)}]
		\item  $\{x_i\}_{i \in D}$ is forward Cauchy if
		\[
		\sup_{i \in D} \inf_{k \geq j \geq i} X(x_j, x_k) = 1.
		\]
		\item  $b$ is a Yoneda limit of $\{x_i\}_{i \in D}$ if for all $y \in X$,
		\[
		X(b, y) = \sup_{i \in D} \inf_{i \leq j} X(x_j, y).
		\]
	\end{enumerate}
\end{definition} 
\begin{definition}
   A weight \(\phi\) is a forward Cauchy ideal of a \([0,1]\)-enriched category $X$ if \[\phi(-) = \sup_{i \in D} \inf_{j \geq i} X(-, x_j)\] for some forward Cauchy net $\{x_i\}_{i \in D}$ of $X$.  All forward Cauchy ideals of \(X\) are denoted by \(\mathcal{I}X\).
\end{definition}
\begin{lem}[\cite{Zhang2024IntroductoryNO}, Lemma 9.11]
    Suppose $\{x_i\}_{i\in D}$ is a forward Cauchy net of a \([0,1]\)-enriched category $X$. Then, an element $b$ of $X$ is a Yoneda limit of $\{x_i\}_{i\in D}$ if and only if $b$ is a colimit of the forward Cauchy weight $\sup_{i\in D} \inf_{j\geq i} X(-,x_j)$.
\end{lem}
\begin{definition}
	A $[0,1]$-enriched category is Yoneda complete if every forward
Cauchy net has a Yoneda limit. Every Yoneda limit is unique up to
isomorphism; if the category is separated, it is unique. In other words, every forward Cauchy ideal of $X$ has a colimit.
\end{definition}
\section{Fuzzy Order on \(\mathbb{E}^1\)}
Defining a fuzzy order on fuzzy numbers can be viewed as defining a fuzzy order on a function space. Early on, a fuzzy order on function spaces was defined based on set inclusion: for any set $X$ and any
functions $f, g \colon X \to [0,1]$ (i.e., membership functions of fuzzy subsets of $X$),
\[
\text{sub}_X (f,g)=\inf_{x\in X}\left(f(x)\rightarrow g(x)\right).
\]
The present definition is inspired by the inclusion-based fuzzy order on function spaces.
\begin{definition}
    Let $u\in \mathbb{E}^1$, and define its left part \( u^\ell \) and right part \( u^r \) as functions  \( \mathbb{R} \to [0,1] \) by:
\[
u^\ell(x) =
\begin{cases}
	0, & x < u_0^-; \\
	u(x), & x \in [u_0^-, u_1^-]; \\
	1, & x > u_1^-.
\end{cases}
\qquad
u^r(x) =
\begin{cases}
	1, & x < u_1^-; \\
	u(x), & x \in [u_1^-, u_0^+]; \\
	0, & x > u_0^+.
\end{cases}
\]
Where \( [u]_0=[u^-_0, u^+_0]\)  denotes the left and right endpoints of the \(0\)-level set; \([u]_1=[u^-_1, u^+_1]\) denotes the  left and right endpoints of the \(1\)-level set.
\end{definition}
\begin{prop}\label{fuzzy number u(x) re}
    Let \( u \in \mathbb{E}^1 \) and \( u = (u^\ell, u^r) \). Then:
\begin{enumerate}[label={\rm(\arabic*)}]
    \item \( u^\ell(x),  u^r(x)\)  are both upper-semicontinuous functions. \( u^\ell(x)\) is a non-decreasing function on \( \mathbb{R} \) and  \(u^r(x)\) is a non-increasing function on \( \mathbb{R} \);
    \item There exist \(x_0,x_1\in \mathbb{R}\) such that $u^\ell(x_0)=u^r(x_1)=0$;
    \item  For every \( x\in \mathbb{R} \), \( u^\ell(x)\vee u^r(x)=1\).
\end{enumerate}
Conversely, suppose that \( (\alpha(x), \beta(x)) \) satisfies \((1)-(3)\), there exists a unique \( u \in \mathbb{E}^1 \) with \( u^\ell = \alpha \) and \( u^r = \beta \).
\end{prop} 

Let \( f: \mathbb{R} \to [0,1] \) be a function. Suppose there exists an \(M\in \mathbb{R} \) such that \(f\) satisfies some of the following properties:
\begin{enumerate}[label={\rm(\arabic*)}]
	\item   For any \(x<-M\), \( f(x) = 0 \) and for any \(x\ge M\), \(f(x)=1\).
    \item \( f(x) \) is non-decreasing.
    \item \( f(x) \) is upper-semicontinuous.
    \item For any \(x\le -M\), \( f(x) = 1 \) and for any \(x>M\), \(f(x)=0\).
     \item \( f(x) \) is non-increasing.
\end{enumerate}

Let \(\mathcal{F}_{[-M,M]}\) denote the set of functions \(f: \mathbb{R} \to [0,1]\) satisfying conditions (1)--(2). Those functions in this set that also satisfy property (3) will be denoted by \(\tilde{\mathcal{F}}_{[-M,M]}\).  
Similarly, let \(\mathcal{G}_{[-M,M]}\) be the set of functions \(g: \mathbb{R} \to [0,1]\) satisfying conditions (4)--(5), and let \(\tilde{\mathcal{G}}_{[-M,M]}\) denote the subset of these functions that also satisfy property (3).

For brevity, the following shorthand notation is adopted:  
\[
\mathcal{F} \equiv \mathcal{F}_{[-M,M]}, \quad \tilde{\mathcal{F}} \equiv \tilde{\mathcal{F}}_{[-M,M]}, \quad \mathcal{G} \equiv \mathcal{G}_{[-M,M]}, \quad \tilde{\mathcal{G}} \equiv \tilde{\mathcal{G}}_{[-M,M]}.
\]  
Unless stated otherwise, this convention is followed.

Throughout this paper, we fix \(M>0\). Let
\(\mathbb{E}^1_{[-M,M]}\) denote the set of all fuzzy numbers whose
support is contained in \([-M,M]\).
For every pair $(f,g)\in\widetilde{\mathcal F}
\times\widetilde{\mathcal G}$ satisfying condition (3) of Proposition \ref{fuzzy number u(x) re}, there exists a unique $u\in\mathbb E^1_{[-M,M]}$
such that \(u^\ell = f\) and \(u^r = g\).

\begin{definition}

Let \( u, v \in \mathbb{E}^1 \) and \( \rightarrow \) be the residual implication of a  left-continuous t-norm \( \& \). Define
\[
P(u, v) = \inf_{x \in \mathbb{R}} \left( (v^\ell(x) \rightarrow u^\ell(x)) \wedge (u^r(x) \rightarrow v^r(x)) \right).
\]

\end{definition}
\begin{remark}
\begin{enumerate}[label={\rm(\arabic*)}]
 \item \(\mathbb{R}\) can be naturally embedded into \((E^1, P)\), where each real number is represented by its characteristic function.
 \begin{align*}
	 u(x)=\chi_{u}(x)=\begin{cases}
	 	1,&x=u;\\
	 	0,& otherwise.
	 \end{cases}, v=\chi_{v}=\begin{cases}
	 1,& x=v;\\
	 0,&otherwise.
 \end{cases}
\end{align*}
 The fuzzy order on \((E^1, P)\) is  a generalization of the natural order on real numbers.
\item Let \(\mathbb{IR}\) denote all the closed intervals of \(\mathbb{R}\).  The natural order on \(\mathbb{IR}\) is defined by  \[[a,b]\le [c,d]\iff a\le c, b\le d. \]
  Each closed interval is represented by its characteristic function. The fuzzy order on \((E^1, P)\)  generalizes the natural order on \(\mathbb{IR}\) as well.
    \item Define \(u\le v\)  if \(P(u,v)=1\). Then 
\begin{align*}
P(u,v)=1 \iff& v^\ell (x)\le u^\ell (x), u^r(x)\le v^r(x), \text{ for all \(x\in \mathbb{R}\)}.\\
\iff& [u]^-_{\alpha}\le [v]^-_{\alpha}, [u]^+_{\alpha}\le [v]^+_{\alpha}, \text{ for all \(\alpha\in [0,1]\)}.
\end{align*}
Thus, the order induced on the base set by our construction coincides with the order defined via  \(\alpha\)-cuts in \cite{Klir1995FuzzySA}. In other words, the latter is precisely the crisp order induced by the present fuzzy order. Since this crisp order is defined componentwise on the $\alpha$-cuts, it is clearly compatible
with the usual fuzzy addition and positive scalar multiplication: if $P(u,v)=1$, then
$P(u \oplus w, v \oplus w)=1$ and $P(\lambda u, \lambda v)=1$ for any $w\in E^1$ and $\lambda>0$.
For the general fuzzy case $P(u,v)\in(0,1)$, we note that full compatibility is not guaranteed
due to the possible mismatch between the t-norm used in the residual implication and that used
in the extension principle for fuzzy arithmetic. The underlying ordering induced by the fuzzy order satisfies the following property.
\end{enumerate}
\end{remark}

\begin{thm}[\cite{Fan2004}, Theorem 2]\label{continuous point of super and infimum}
	Suppose that \(\{u_i\}_{i\in D}\) is a family of \(\mathbb{E}^1_{[-M,M]}\), \(s\) and \(m\) are the colimit of the conical weight \(\phi\) and the limit of the conical coweight \(\psi\), respectively. Then,
	\begin{align*}
		s^\ell(x)=\inf_{i\in D} u^\ell_i(x), \quad
		s^r(x)=\begin{cases}
			\sup_{i\in D}u^r_i(x),& x\in [-M,M] \backslash Dis(s)\\
			\lim_{x_n\nearrow x}\sup_{i\in D}u^r_i(x_n),& otherwise.
		\end{cases},
	\end{align*}
	where \(Dis (s)=\{x\mid\text{\(x\) is a discontinuous point of } s \}\) is a countable set.
	\begin{align*}
		m^\ell(x)=\begin{cases}
			\sup_{i\in D}u^\ell_i(x),& x\in [-M,M]\backslash Dis(m)\\
				\lim_{x_n\searrow x} \sup_{i\in D}u^\ell_i(x_n), & otherwise.
		\end{cases}, \quad
		m^r(x)= \inf_{i\in D} u^r_i(x),
	\end{align*}
	where  \(Dis (m)=\{x\mid \text{\(x\) is a discontinuous point of } m \}\) is a countable set.
\end{thm}

\begin{prop}
	 Let $\&$ be a left-continuous t-norm. $P(u,v)$ is a fuzzy order on $\mathbb{E}^1$ and  is separated.
\end{prop}
\begin{proof}
Reflexivity is clear: \( P(u, u) = 1 \). For transitivity, take \( u, v, w \in \mathbb{E}^1 \). 
 Using property (3) of the implication, we verify:
\begin{align*}
    P(u, w) \& P(w, v) &=\inf_{x\in \mathbb{R}} \left[(w^\ell(x) \rightarrow u^\ell(x)) \wedge (u^r(x) \rightarrow w^r(x)) \right] \& \inf_{x\in \mathbb{R}} \left[ (v^\ell(x) \rightarrow w^\ell(x)) \wedge (w^r(x) \rightarrow v^r(x))\right] \\
    &\le \inf_{x\in \mathbb{R}} \left\{\left[(w^\ell(x) \rightarrow u^\ell(x)) \wedge (u^r(x) \rightarrow w^r(x)) \right]\& \left[(v^\ell(x) \rightarrow w^\ell(x)) \wedge (w^r(x) \rightarrow v^r(x))\right]\right\} \\
    &\le \inf_{x\in \mathbb{R}} \left\{\left[(w^\ell(x) \rightarrow u^\ell(x)) \& (v^\ell(x) \rightarrow w^\ell(x)) \right]\wedge \left[(u^r(x) \rightarrow w^r(x))  \& (w^r(x) \rightarrow v^r(x))\right]\right\}\tag{Since \(\wedge\) is the largest t-norm }\\
    &\le \inf_{x\in \mathbb{R}} (v^\ell(x) \rightarrow u^\ell(x)) \wedge (u^r(x) \rightarrow v^r(x))=P(u,v).
\end{align*}
For antisymmetry, suppose \(P(u,v)=P(v,u)=1\). Then \(u=v\).
\end{proof}
\begin{thm}\label{continuous and except for countable point}
	 	Let $\&$ be a  left-continuous  t-norm. For any \( u, v \in \mathbb{E}^1 \), let \[ D = \{x \mid x \text{ is a point of discontinuity of } u \text{ or } v \}.\]  Then
	\[
	P'(u, v) := \inf_{x \in \mathbb{R}\setminus D} \left[ (v^\ell(x) \rightarrow u^\ell(x)) \wedge (u^r(x) \rightarrow v^r(x)) \right] = P(u, v).
	\]
\end{thm}
\begin{proof}
The inequality \(P'(u,v)\ge P(u,v)\) is clear. It remains to show that \(P'(u,v)\le P(u,v) \).
Since every fuzzy number is monotone on each side, it has
at most countably many discontinuity points. Hence \(D\) is at most
countable and \(\mathbb R\setminus D\) is dense.

Let
\[
A(x)=v^\ell(x)\rightarrow u^\ell(x),\qquad
B(x)=u^r(x)\rightarrow v^r(x),
\]
and
\[
H(x)=A(x)\wedge B(x).
\]
Let \(L=P(u,v)\). If \(L=1\), then
\(P'(u,v)=1=P(u,v)\). Assume \(L<1\), and suppose, for
contradiction, that \(L<P'(u,v)\). Choose
\[
0<\gamma<\min\{1-L,P'(u,v)-L\}.
\]
By the definition of \(L\), there exists \(x_0\in\mathbb R\) such that
\[
H(x_0)<L+\gamma.
\]

Suppose first that
\[
A(x_0)<L+\gamma.
\]
Since \(u^\ell\) and \(v^\ell\) are non-decreasing and
upper-semicontinuous, they are right-continuous. Hence we may choose
a sequence \(x_n\searrow x_0\) with \(x_n\notin D\), and
\[
v^\ell(x_n)\to v^\ell(x_0),\qquad
u^\ell(x_n)\to u^\ell(x_0).
\]
Since the residuum is jointly upper semicontinuous by Proposition \ref{left-continuous is jointly upper semicontinuous function},
\[
\limsup_{n\to\infty}A(x_n)\le A(x_0)<L+\gamma.
\]
Therefore \(A(x_n)<L+\gamma\) for some \(n\), and hence
\[
P'(u,v)\le H(x_n)\le A(x_n)<L+\gamma<P'(u,v),
\]
a contradiction.

The case \(B(x_0)<L+\gamma\) is analogous. Since \(u^r\) and
\(v^r\) are non-increasing and upper-semicontinuous, they are
left-continuous. Choose \(x_n\nearrow x_0\) with \(x_n\notin D\). Then
\[
\limsup_{n\to\infty}B(x_n)\le B(x_0)<L+\gamma,
\]
which again yields
\[
P'(u,v)< L+\gamma<P'(u,v),
\]
a contradiction.

Thus \(P'(u,v)=P(u,v)\).
\end{proof}

Theorem \ref{continuous and except for countable point} concerns any two elements \(u\) and \(v\) in \( \mathbb{E}^1 \). It shows that the value \(P'(u, v)\) coincides with \( P(u, v) \), where \(D\) is the set of discontinuity points of \(u\) or \(v\). In other words, if we exclude the points of discontinuity and evaluate only the "well-behaved" points, the resulting infimum is unchanged. 
 
The implication expressions of common t-norms and their corresponding fuzzy order values are listed in Table \ref{tab:fuzzy_order_tnorm}.

\begin{exmp}
Consider the triangular fuzzy numbers \( u = (0,1,3) \) and \( v = (-1,2,2.5) \), illustrated in Figure \ref{fuzzy number u,v}. The value of their fuzzy order \( P(u, v) \) is computed in Table \ref{tab:fuzzy_order_tnorm}. The result shows that \( P(u, v) \) is sensitive to the choice of t-norms. The Łukasiewicz t-norm  and Nilpotent Minimum reflect a non-trivial fuzzy order between \( u \) and \( v \), whereas  the other two t-norms yield \( P(u, v) = 0 \). This is because the Gödel and product implications drop to zero as soon as \(v^\ell(x) > u^\ell(x)=0\) for some $x$. These observations indicate that the distribution and structural properties of fuzzy numbers influence the value of the fuzzy ordering relation. 

From an applied perspective, the choice of the t-norm reflects the decision-maker's tolerance for how strictly the containment between
membership functions should be evaluated. The {\L}ukasiewicz and Nilpotent Minimum t-norms permit a graded comparison even when the supports of two fuzzy numbers partially diverge, whereas the G\"odel and product t-norms impose a stricter criterion: $P(u,v)=0$ whenever the support of one fuzzy number extends beyond that of the other at any point. Consequently, the selection of a specific t-norm should be guided by the desired degree of tolerance to support divergence in the intended application. It should be noted, however, that when $P(u,v)=1$ the fuzzy order degenerates to the crisp $\alpha$-cut componentwise order, which is independent of the choice of t-norm.
	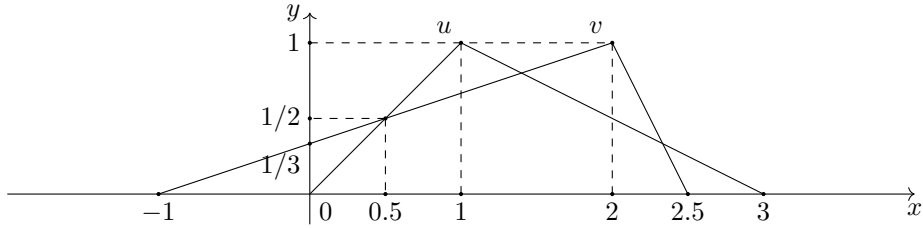
\begin{figure}[htbp]	
	\centering
	\begin{minipage}[t]{0.8\textwidth}
		\centering
		\begin{tikzpicture}	[scale=2]
			\draw[->] (-2,0)--(4,0);
			\draw[->] (0,-0.2)--(0,1.2); 
			\draw (0,0)--(1,1)--(3,0);
			\draw (-1,0)--(2,1)--(2.5,0);
			\draw[dashed](1,0)--(1,1);
			\draw[dashed](2,1)--(2,0);
			\draw[dashed](0,1)--(2,1);
			\draw[dashed](1/2,1/2)--(1/2,0);
			\draw[dashed](1/2,1/2)--(0,1/2);
			\draw[fill] (0,1/3) circle(.01);
			\draw[fill] (0,1) circle(.01);
			\draw[fill] (-1,0) circle(.01);
			\draw[fill] (2,0) circle(.01);
			\draw[fill] (1,0) circle(.01);
			\draw[fill] (3,0) circle(.01);
			\draw[fill] (2.5,0) circle(.01);
			\draw[fill] (2,1) circle(.01);
			\draw[fill] (1,1) circle(.01);
			\draw[fill] (1/2,1/2) circle(.01);
			\draw[fill] (0,1/2) circle(.01);
			\draw[fill] (1/2,0) circle(.01);
			\node (a) at (0,1) [ left] {$1$};
			\node (b) at (1,1) [above left] {$u$};
			\node (c) at (2,1) [above left] {$v$};
			\node (d) at (-1,0) [below ] {$-1$};
			\node (e) at (0,0) [below right] {$0$};
			\node (f) at  (1,0) [below]{$1$};
			\node (g) at  (2.5,0) [below]{$2.5$};
			\node (h) at  (2,0) [below]{$2$};
			\node (i) at  (3,0) [below]{$3$};
			\node (j) at  (0,1/3) [below left]{$1/3$};
			\node (k) at  (0,1/2) [left]{$1/2$};
			\node (l) at  (1/2,0) [below]{$0.5$};
        \node (x) at (4,0) [ below] {$x$};
        \node (y) at (0,1.2) [left] {$y$};
		\end{tikzpicture} 
	\end{minipage}
	\caption{ The graph of two  triangular fuzzy numbers $u=(0,1,3),v=(-1,2,2.5)$.  }\label{fuzzy number u,v}
\end{figure}
\begin{table}[!htbp]
	\centering
	\caption{Fuzzy Order \(P(u, v)\) of Triangular Fuzzy Numbers Under Different T-norms.}
	\label{tab:fuzzy_order_tnorm}
	\begin{tabular}{lcc}
		\toprule
		\multicolumn{1}{c}{T-norm Type} & \multicolumn{1}{c}{Residual Implication \(x \to y\)} & \multicolumn{1}{c}{Fuzzy Order \(P(u, v)\)} \\
		\midrule
		Gödel & 
		\(\begin{cases} 
			1, & x \leq y; \\ 
			y, & x > y .
		\end{cases}\) & 
			\( 1/3 \rightarrow 0\)=0. \\
		\midrule
		Product & 
		\(\begin{cases} 
			1, & x \leq y; \\ 
			\displaystyle y/x, & x > y. 
		\end{cases}\) (\( \frac{0}{0}=1 \)) & 
		\( 1/3 \rightarrow 0\)=0. \\
		\midrule
		Łukasiewicz & 
		\(\min\{1, 1 - x + y\}\) & 
		\(\displaystyle 1/3 \rightarrow 0=2/3 \).\\
		\midrule
		Nilpotent Minimum & 
		\(\begin{cases} 
			1, & x \leq y; \\ 
			\max\{1 - x, y\}, & x > y. 
		\end{cases}\) & 
		\(\lim_{x\nearrow 0.5}  (x+1)/3 \rightarrow x  =1/2.\) \\
		\bottomrule
	\end{tabular}
\end{table}
\end{exmp}

\section{Completeness of the Fuzzy Order}
With the fuzzy order \(P\) defined, its order-theoretic completeness properties are investigated. Specifically, we investigate the conical completeness and conical cocompleteness of its uniformly bounded subspace \(\mathbb{E}^1_{[-M,M]}\), as well as the completeness properties of \(\mathbb{E}^1\). These results are crucial for revealing the structural complexity of the space. 
\begin{prop}\label{upper-semicontinuous except for countable point eq}
Let \(g\in\mathcal{G}\), and let \(\widetilde{g}(x)=\inf_{x'<x}g(x')\) be its upper-semicontinuous modification, so that \(\widetilde{g}\in\mathcal{\widetilde{G}}\). If $\&$ is a continuous t-norm, then for any \( g' \in \mathcal{\widetilde{G}} \), 
\[
\inf_{x \in [-M,M]} (g(x) \rightarrow g'(x)) = \inf_{x \in [-M,M]} (\widetilde{g}(x) \rightarrow g'(x)).
\]
A symmetric statement holds for \( f\in \mathcal{F}, f'\in \mathcal{\widetilde{F}}\).
\end{prop}
  \begin{proof}
  $\inf_{x \in [-M,M]} (g(x) \rightarrow g'(x))=g(-M) \rightarrow g'(-M) \bigwedge \inf_{x\in (-M,M]} (g(x) \rightarrow g'(x))$. Since $g(-M)=\widetilde{g}(-M)$, it suffices to prove that
   \[
 \inf_{x \in (-M,M]} (g(x) \rightarrow g'(x)) = \inf_{x \in (-M,M]} (\widetilde{g}(x) \rightarrow g'(x)).
  \]
 Since $\widetilde{g}(x)\ge g(x)$,  $ \widetilde{g}(x)\rightarrow g'(x)\le g(x)\rightarrow g'(x)$ for every $x\in (-M,M]$. Then, \[
  \inf_{x \in (-M,M]} (\widetilde {g}(x) \rightarrow g'(x))\le  \inf_{x \in (-M,M]} (g(x) \rightarrow g'(x)).
 \]
 For the reverse inequality, let $ \inf_{x \in (-M,M]} (g(x) \rightarrow g'(x))=t$. By the adjunction property, this is equivalent to \( g(x) \& t \leq g'(x) \) for all \( x \in(-M,M]\). Since \( g '\) is upper-semicontinuous, \( g'(x) = \inf_{y < x} g'(y) \). Using the continuity of \( \& \) and definition of \( \widetilde{g}\),
 \[
 \widetilde{g}(x) \& t = \left( \inf_{y < x} g(y) \right) \& t = \inf_{y < x} (g(y) \& t) \leq \inf_{y < x} g'(y)=g'(x).
 \]
 Applying adjunction again yields \( t \leq \widetilde{g}(x)\rightarrow g'(x) \) for all \( x \in (-M,M]\), so \( t \leq \inf_{x\in(-M,M]} \widetilde{g}(x)\rightarrow g'(x) \). A similar argument shows that a symmetric statement holds for \( f \in \mathcal{F}, f' \in \mathcal{\widetilde{F}} \). 
 \end{proof}

\begin{prop}\label{conically cocomplete and conically complete in E}
		Let $\&$ be a  continuous  t-norm. $(\mathbb{E}^1_{[-M,M]},P)$ is  conically cocomplete and  conically complete.
\end{prop}
\begin{proof}
	 Let \(\phi \) be a  conical weight  with \(A=\{u_i\}_{i\in D}\subseteq \mathbb{E}^1_{[-M,M]}\). Its colimit is \[ s=(\inf_{i\in D} u_i^\ell(x), \widetilde{\sup_{i\in D} u_i^r}(x) ), \text{ \(\widetilde{\sup_{i\in D} u_i^r}\) denotes its upper-semicontinuous modification}.\]
	 Clearly, \(s\) is a fuzzy number in \(\mathbb{E}^1_{[-M,M]}\).
	 We show that \(s\) is a colimit of  conical weight \(\phi\).
	 For any \(w\in\mathbb{E}^1_{[-M,M]}\),  
	 \begin{align*}
	 	\inf_{e\in \mathbb{E}^1_{[-M,M]}} \Bigl(\phi (e)\rightarrow P(e,w) \Bigl) =&\inf_{e\in \mathbb{E}^1_{[-M,M]}}  \Bigl(\sup_{i\in D} P(e,u_i) \rightarrow P(e,w)\Bigl) \\
	 	=& \inf_{e\in \mathbb{E}^1_{[-M,M]}} \inf_{i\in D} (P(e,u_i) \rightarrow P(e,w) )\\
	 	=& \inf_{i\in D} P(u_i,w) \\
	 	= &\inf_{i\in D} \inf_{x\in [-M,M]}( w^\ell(x)\rightarrow u_i^\ell(x))\wedge (u_i^r(x)\rightarrow w^r(x))  \\
	 	=& \inf_{x\in [-M,M]}\Bigl( ( w^\ell(x)\rightarrow  \inf_{i\in D}u^\ell_i(x) )\wedge (\sup_{i\in D}u_i^r(x)\rightarrow w^r(x) ) \Bigl)\\
	 	=&P(s,w). \tag{ By Proposition \ref{upper-semicontinuous except for countable point eq}}
	 \end{align*}
 Hence, \( \mathbb{E}^1_{[-M,M]}\) is conically cocomplete. Let \(\psi \) be a conical coweight with \(B=\{u_i\}_{i\in D}\subseteq \mathbb{E}^1_{[-M,M]} \).
 Similarly, we  can prove that \( \mathbb{E}^1_{[-M,M]}\) is conically complete. 
 The limit of conical coweight of \(\psi=\sup_{i}P(u_i,-)\) is \[ m=(\widetilde{\sup_{i\in D} u_i^\ell}(x), \inf_{i\in D} u_i^r(x) ).\] 
\end{proof}
\begin{thm}\label{Yoneda limit of P}
Let $\&$ be a  continuous t-norm. Then $(\mathbb{E}^1_{[-M,M]},P)$ is  Yoneda complete.    
\end{thm}
\begin{proof}
Let $\phi=\sup_{i\in D}\inf_{j\ge i}P(-,u_j)$ be a forward Cauchy ideal, where
$\{u_j\}_{j\in D}$ is a forward Cauchy net. By Proposition~\ref{conically cocomplete and conically complete in E}, for each
$i\in D$ the meet $m_i:=\inf_{j\ge i}u_j$ exists and satisfies
\[
P(-,m_i)=\inf_{j\ge i}P(-,u_j).
\]
Hence
\[
\phi=\sup_{i\in D}P(-,m_i),
\]
which is a conical weight. Since the family $\{m_i\}_{i\in D}$ is directed (as $i$
grows, $\{j:j\ge i\}$ shrinks, so $m_i$ increases), its join $u:=\sup_{i\in D}m_i$
exists, and by 
Proposition~\ref{conically cocomplete and conically complete in E},
\[
\operatorname{colim}\phi=\sup_{i\in D}m_i=:u.
\]
Finally, by Theorem~\ref{continuous point of super and infimum},
\[
m_i^\ell=\widetilde{\sup_{j\ge i}u_j^\ell},\qquad m_i^r=\inf_{j\ge i}u_j^r,
\]
so
\[
u^\ell=\inf_{i\in D}m_i^\ell=\inf_{i\in D}\widetilde{\sup_{j\ge i}u_j^\ell},
\qquad
u^r=\widetilde{\sup_{i\in D}m_i^r}=\widetilde{\sup_{i\in D}\inf_{j\ge i}u_j^r}.
\]
Thus $\operatorname{colim}\phi=u$ exists with the displayed branches, and
$(\mathbb{E}^1_{[-M,M]},P)$ is Yoneda complete.
\end{proof}

\begin{cor}\label{complete lattice}
	\((\mathbb{E}^1_{[-M,M]})_0 \) is a complete lattice.
\end{cor}
\begin{proof}
	By  Proposition \ref{conically cocomplete and conically complete in E}, the colimit of conical weight of \(\phi=\sup_{i\in D} P(-,u_i)\) is the supremum of \(u_i\) in  the underlying order. 
\end{proof}
Following example show that \(\mathbb{E}^1_{(-\infty,M]}\) is not Yoneda complete under Łukasiewicz t-norm, where \(\mathbb{E}^1_{(-\infty,M]}\) denotes the set of all fuzzy numbers with supports contained in \((-\infty,M]\).
\begin{exmp}
    \begin{align*}
        u_n=\begin{cases}
            e^x,-n \le x\le 0;\\
            0, others.
        \end{cases}
    \end{align*}
    where \(n\in \mathbb{Z}^+\).
    \(u_n\) is a forward Cauchy net in \(\mathbb{E}^1_{(-\infty,0]}\), but it has no Yoneda limit. Hence, \(\mathbb{E}^1_{(-\infty,M]}\) is not Yoneda complete under Łukasiewicz t-norm.
\end{exmp}

\begin{prop}
Let \(\&\) be a left-continuous t-norm. Then \((\mathbb E^1,P)\) is neither
finitely conically cocomplete, nor finitely conically complete, nor tensored, nor cotensored.
\end{prop}
\begin{proof}
   The underlying ordered set \((\mathbb E^1)_0\) has neither a least nor a
greatest element. Hence the empty conical weight (resp.\ empty conical
coweight) has no colimit (resp.\ limit), so \((\mathbb E^1,P)\) is neither
finitely conically cocomplete nor finitely conically complete. Moreover, since \(0\to t=1\),
a tensor \(0\otimes x\), if it existed, would be the least element, and a
cotensor \(0\multimap y\), if it existed, would be the greatest element;
neither exists, so \((\mathbb E^1,P)\) is neither tensored nor cotensored.
\end{proof}
\begin{cor}
    Let \(\&\) be a left-continuous t-norm. Then \((\mathbb{E}^1,P) \) is not complete.
\end{cor}
\begin{proof}
    By Theorem \ref{complete category}.
\end{proof}
\begin{prop}
Let \(\&\) be a left-continuous t-norm such that \(r\to0>0\) for some
\(r\in(0,1)\) (i.e.\ \(r\) is a zero divisor). Then
\((\mathbb E^1_{[-M,M]},P)\) is neither tensored nor cotensored.
\end{prop}
\begin{proof}
We use the following elementary consequences of adjunction:
\[
r\to q=1\iff r\le q,\qquad q<r\Longrightarrow r\to q<1,\qquad 1\to0=0.
\]

Let \(u=(-M,0,M)\) be the triangular fuzzy number, so that
\(u^\ell(x)=\frac{x+M}{M}\) for \(-M\le x\le0\) and
\(u^r(x)=\frac{M-x}{M}\) for \(0\le x\le M\).

\emph{Tensor.} Suppose that \(w=r\otimes u\) exists, i.e.\
\(P(w,v)=r\to P(u,v)\) for all \(v\). Taking \(v=u\) gives \(P(w,u)=1\);
taking \(v=w\) gives \(1=P(w,w)=r\to P(u,w)\), whence \(P(u,w)\ge r\).
Therefore
\[
u^\ell\le w^\ell\le r\to u^\ell,\qquad
r\mathbin{\&}u^r\le w^r\le u^r.
\]

Choose \(\varepsilon\in(0,r)\), and set
\(a_-=M(\varepsilon-1)\), \(b_-=M(r-1)\), so that
\(-M<a_-<b_-<0\). Let \(v=\chi_{a_-}\). Since \(u^r(x)=1\) and
\(v^r(x)=0\) for \(x\in(a_-,0)\), we have \(P(u,v)=0\), and the tensor
identity yields \(P(w,v)=r\to0>0\).

On the other hand, for \(x\in(a_-,b_-)\) we have
\(u^\ell(x)=\frac{x+M}{M}<r\), hence
\(w^\ell(x)\le r\to u^\ell(x)<1\). As \(w^\ell(x)\vee w^r(x)=1\), we get
\(w^r(x)=1\). Moreover \(x>a_-\), so \(v^\ell(x)=1\) and \(v^r(x)=0\).
Thus
\[
P(w,v)\le\bigl(v^\ell(x)\to w^\ell(x)\bigr)\wedge\bigl(w^r(x)\to v^r(x)\bigr)
=w^\ell(x)\wedge(1\to0)=0,
\]
contradicting \(P(w,v)=r\to0>0\). Hence \(r\otimes u\) does not exist,
and the category is not tensored.

\emph{Cotensor.} Dually, reflecting via \(x\mapsto -x\) (which fixes the
symmetric fuzzy number \(u\)), \(r\multimap u\) does not exist either. Therefore, \((\mathbb E^1_{[-M,M]},P)\) is neither tensored nor cotensored.
\end{proof}
\begin{cor}
	Let \(\&\) be a left-continuous t-norm such that \(r\to0>0\) for some
\(r\in(0,1)\). Then
\((\mathbb E^1_{[-M,M]},P)\) is not complete.
\end{cor}
\begin{proof}
    By Theorem \ref{complete category}.
\end{proof}
The non-existence of tensors and cotensors proved above relies on the
zero-divisor condition \(r\to0>0\). It is natural to ask whether this
hypothesis can be removed. The following example shows that it cannot: for
the Gödel t-norm, which has no zero divisors, both the tensor and the
cotensor exist.
\begin{exmp}[Tensors and cotensors for the Gödel t-norm]\label{ex:Godel-tensor}
Let \(\&=\wedge\) be the Gödel t-norm.
Fix \(r\in(0,1)\) and \(u\in\mathbb E^1_{[-M,M]}\), and write
\([u]_r=[u^-_r,u^+_r]\).

\emph{Tensor.} Define \(w=r\otimes u\) by its branches
\[
w^\ell=r\to u^\ell,\qquad
w^r(x)=\begin{cases}1,&x\le u^-_r,\\r\wedge u^r(x),&x>u^-_r.\end{cases}
\]
\emph{Cotensor.} Define \(c=r\multimap u\) by its branches
\[
c^\ell(x)=\begin{cases}r\wedge u^\ell(x),&x<u^+_r,\\1,&x\ge u^+_r,\end{cases}
\qquad
c^r=r\to u^r.
\]
 By proposition~\ref{fuzzy number u(x) re}, we verify that \(w\) and \(c\) are fuzzy numbers.

\emph{Verify identity.} Since \(w^\ell=r\to u^\ell\) and
\(w^r\ge r\wedge u^r\), by residuation we have
\(v^\ell\to w^\ell=r\to(v^\ell\to u^\ell)\) and
\(w^r\to v^r\le(r\wedge u^r)\to v^r=r\to(u^r\to v^r)\), whence
\[
P(w,v)\le(r\to A)\wedge(r\to B)=r\to P(u,v),
\]
where \(A=\inf_x(v^\ell\to u^\ell)\), \(B=\inf_x(u^r\to v^r)\). For the reverse inequality, let \(s=P(u,v)=A\wedge B\).

\emph{Case 1: \(r\le s\).} Then \(r\to s=1\), so it suffices to prove
\(P(w,v)=1\). Since \(A\ge s\ge r\), we have \(v^\ell(x)\to u^\ell(x)\ge r\)
for every \(x\), whence
\(v^\ell(x)\to w^\ell(x)=r\to\bigl(v^\ell(x)\to u^\ell(x)\bigr)=1\), that is,
\(v^\ell(x)\le w^\ell(x)\). For the right branch, if \(x>u^-_r\), then
\(u^r(x)\to v^r(x)\ge B\ge r\), so
\(w^r(x)\to v^r(x)=r\to\bigl(u^r(x)\to v^r(x)\bigr)=1\), that is,
\(w^r(x)\le v^r(x)\). If \(x\le u^-_r\), then \(w^r(x)=1\), and we claim that
\(v^r(x)=1\). To see this, note first that \(v_1^-\ge u^-_r\): for \(y<u^-_r\)
one has \(u^\ell(y)<r\); if \(y>v_1^-\), then \(v^\ell(y)=1\), which would give
\(v^\ell(y)\to u^\ell(y)=u^\ell(y)<r\), contradicting
\(v^\ell(y)\to u^\ell(y)\ge r\). Thus \(v_1^-\ge u^-_r\), and consequently
\(v^r(x)=1\) for every \(x\le u^-_r\). Hence \(P(w,v)=1=r\to s\).

\emph{Case 2: \(r>s\).} Then \(r\to s=s\). For the Gödel implication, if
\(t\ge s\), then \(r\to t\ge s\). Hence, for the left branch,
\(v^\ell(x)\to w^\ell(x)=r\to\bigl(v^\ell(x)\to u^\ell(x)\bigr)\ge s\), because
\(v^\ell(x)\to u^\ell(x)\ge A\ge s\). For the right branch, if \(x>u^-_r\),
then
\(w^r(x)\to v^r(x)=r\to\bigl(u^r(x)\to v^r(x)\bigr)\ge s\), since
\(u^r(x)\to v^r(x)\ge B\ge s\). If \(x\le u^-_r\), then \(w^r(x)=1\), and
because \(u^-_r\le u_1^-\), we have \(u^r(x)=1\), so
\(w^r(x)\to v^r(x)=1\to v^r(x)=v^r(x)=u^r(x)\to v^r(x)\ge B\ge s\).
Therefore \(P(w,v)\ge s=r\to s\).
 The cotensor identity
\(P(v,r\multimap u)=r\to P(v,u)\) is verified dually.
\end{exmp}
\begin{thm}\label{completeness iff Godel}
Let \(\&\) be a continuous t-norm. The following are equivalent:
\begin{enumerate}[label={\rm(\arabic*)}]
    \item \((\mathbb E^1_{[-M,M]},P)\) is complete;
    \item \((\mathbb E^1_{[-M,M]},P)\) is cocomplete;
    \item \((\mathbb E^1_{[-M,M]},P)\) is tensored;
    \item \((\mathbb E^1_{[-M,M]},P)\) is cotensored;
    \item \(\&=\wedge\).
\end{enumerate}
\end{thm}

\begin{proof}
By Proposition~\ref{conically cocomplete and conically complete in E} and
Theorem~\ref{complete category}, the first four properties are equivalent,
so it suffices to prove \(\text{tensored}\iff\&=\wedge\).

\emph{($\Leftarrow$).} This is Example~\ref{ex:Godel-tensor}.

\emph{($\Rightarrow$).} Assume tensored. Let \(u=(-M,0,M)\), fix
\(r\in(0,1)\), and put \(w=r\otimes u\), so that \(P(w,v)=r\to P(u,v)\) for
all \(v\). Taking \(v=u\) gives \(P(w,u)=1\); taking \(v=w\) gives
\(1=r\to P(u,w)\), hence \(P(u,w)\ge r\). Therefore
\[
u^\ell\le w^\ell\le r\to u^\ell,\qquad r\mathbin{\&}u^r\le w^r\le u^r.
\]

Take \(p\in(0,r)\), set \(c=M(p-1)\), and let \(u_r^-=M(r-1)\) be the left
endpoint of the \(r\)-cut of \(u\), so \(-M<c<u_r^-<0\). Define
\(v\in\mathbb E^1_{[-M,M]}\) by
\[
v^\ell(x)=\begin{cases}0,&x<c,\\1,&x\ge c,\end{cases}\qquad
v^r(x)=\begin{cases}1,&x\le c,\\p,&c<x\le M,\\0,&x>M.\end{cases}
\]
Thus \(v(c)=1\), \(v=p\) on \((c,M]\), and \(v\) is supported in
\([c,M]\subseteq[-M,M]\).

We claim \(P(u,v)=p\). Indeed, \(\inf_x(v^\ell(x)\to u^\ell(x))=u^\ell(c)=p\),
and \(\inf_x(u^r(x)\to v^r(x))=p\), since \(u^r=1\) on \([c,0]\), \(v^r=p\)
on \((c,M]\), and \(p\le q\to p\) for all \(q\in[0,1]\).

Now take \(x\in(c,u_r^-)\), so \(u^\ell(x)<r\). Then
\(w^\ell(x)\le r\to u^\ell(x)<1\), hence \(w^r(x)=1\), while \(v^r(x)=p\).
Therefore
\[
P(w,v)\le w^r(x)\to v^r(x)=1\to p=p.
\]
On the other hand, \(P(w,v)=r\to P(u,v)=r\to p\), so \(r\to p\le p\); with
\(p\le r\to p\) (from \(r\mathbin{\&}p\le p\)) this gives \(r\to p=p\).

Since \(r\to p=p\), for \(z>p\) we have \(r\mathbin{\&z}>p\) (else
\(z\le r\to p=p\)). Taking \(z=r>p\) gives \(r\mathbin{\&}r>p\); letting
\(p\nearrow r\) yields \(r\mathbin{\&}r=r\), so every element of \([0,1]\) is
idempotent. Finally, if \(a\le b\), then
\(a=a\mathbin{\&}a\le a\mathbin{\&}b\le a\mathbin{\&}1=a\), whence
\(a\mathbin{\&}b=a=\min\{a,b\}\), i.e.\ \(\&=\wedge\).
\end{proof}

\section{Continuity of Fuzzy Order on \((\mathbb{E}^1_{[-M,M]},P)\)}
This section focuses on the continuity of the fuzzy order $P$ on the space of uniformly bounded-support fuzzy numbers $\mathbb{E}^1_{[-M,M]}$. Continuity is a property essential for approximation and computation, and it requires a thorough investigation.  Our approach begins by analyzing the continuity of left and right branch function spaces ($\widetilde{\mathcal{F}}$ and $\widetilde{\mathcal{G}}$) under fuzzy order.  We then prove that \((\mathbb{E}^1_{[-M,M]},P)\) is a [0,1]-domain if and only if the continuous t-norm satisfies the (S) condition. We begin  by giving the definition of continuity for a \([0,1]\)-enriched category.
\begin{definition}
	 Suppose \( X \) is a \([0,1]\)-enriched category. If \( \operatorname{colim}: \mathcal{I}X \to X \) has a left adjoint, then we say that \( X \) is continuous. In other words, \( X \) is continuous if there exists a string of adjunctions
	\[
	 \twoheaddownarrow \dashv \operatorname{colim} \dashv \mathrm{y}: X \to \mathcal{I}X.
	\]
\end{definition}
A Yoneda complete and continuous \([0,1]\)-enriched category is called a \([0,1]\)-enriched domain.
Let \( X \) be a \([0,1]\)-enriched category. The way below distributor of \( X \) refers to the distributor \( \mathfrak{w}: X \to X \) given by
\[
\mathfrak{w}(y, x) = \inf_{\phi \in \mathcal{I}X} \left( X(x, \operatorname{colim} \phi) \to \phi(y) \right).
\]

\begin{lem}[\cite{Zhang2024IntroductoryNO}, Proposition 17.3]\label{way-below}
	 A \([0,1]\)-enriched category \( X \) is continuous if and only if for all \( x \in X \), the weight \( \mathfrak{w}(-, x) \) belongs to \( \mathcal{I}X \) with \( x \) being a colimit. In this case, \( \twoheaddownarrow x = \mathfrak{w}(-, x) \).
\end{lem}
 Following  \cite{Lai2019CompletelyDE}, we present some conclusions concerning the continuity of the fuzzy order.
 \begin{prop}[\cite{Lai2019CompletelyDE}]\label{implication operator is continuous at every point}
 	Suppose $\mathbin{\&}$ is a continuous t-norm on $[0,1]$. For each $p \in [0,1]$,
 	let $p^-$ be the greatest idempotent element in $[0,p]$ and let $p^+$ be the least idempotent
 	element in $[p,1]$. Then the following are equivalent:
 	
 	\begin{enumerate}[label={\rm(\arabic*)}]
 		\item For each non-idempotent element $p \in [0,1]$, the restriction of $\mathbin{\&}$ on $[p^-,p^+]$
 		is isomorphic to the product t-norm on \([0,1]\) whenever $p^- > 0$.
 		
 		\item The implication operator \(\to: [0,1]^2 \to [0,1]\) is continuous at every point
 		off the diagonal $\{(x,x) \mid x \in [0,1]\}$.
 		
 		\item For each $p \in (0,1]$, the function $p \to -: [0,1] \to [0,1]$ is continuous on the interval $[0,p)$.
 	\end{enumerate}
 \end{prop}
 A continuous t-norm that satisfies one of the conditions in Proposition  \ref{implication operator is continuous at every point} is said to satisfy the (S) condition.
\begin{thm}[\cite{Lai2019CompletelyDE}, Theorem 6.4]\label{continuous= (S) condition}
	Let $\mathbin{\&}$ be a continuous t-norm. Then the following
	statements are equivalent:
	\begin{enumerate}[label={\rm(\arabic*)}]
	
		\item The $[0,1]$-enriched category $([0,1], \rightarrow)$ is continuous.
		\item A continuous t-norm that satisfies the (S) condition.
	\end{enumerate}
\end{thm}

\begin{definition}
	Let  \(\&\) be a  left-continuous t-norm. For any \(f,g\in \tilde{\mathcal{G} }\),
	\[P(f,g)=\inf_{x\in [-M,M]} \left(f(x)\rightarrow g(x)\right).\]
\end{definition}
\begin{cor}
	Let  \(\&\) be a continuous t-norm. \((\tilde{\mathcal{G} },P)\) is a separated, conically cocomplete  and  conically complete \([0,1]\)-enriched category.
\end{cor}
\begin{proof}
Separatedness follows directly from the definition of \(P\), while conical cocompleteness and conical completeness follow by applying the same argument as in the proof of Proposition~\ref{conically cocomplete and conically complete in E} to the right-branch space \(\widetilde{\mathcal G}\).
\end{proof}
\begin{prop}\label{way-below of crisp order}
\((\tilde{\mathcal{G}})_0=(\tilde{\mathcal{G}},\le)\) is a continuous lattice. For any \(f\in(\tilde{\mathcal{G}},\le)\) and any \(0<p<1\), if \(r\) is a continuous point of \(f\), then \(\mathcal{E}^r_{p,r}\ll f\) if and only if \(p<f(r)\).
\end{prop}
\begin{proof}
We define:
\[
\mathcal{E}^r_{p,r}(x)=\begin{cases}
	1, & x\le -M;\\
	p, & -M<x\le r;\\
	0, & x>r.
\end{cases}
\]

\(\implies\) Suppose \(\mathcal{E}^r_{p,r}\ll f\). If \(r=-M\), then \(f(r)=f(-M)=1>p\), and the claim follows. Assume \(r>-M\). We construct a family of functions \(f_n\):
\[
f_n(x)=\begin{cases}
	1, & x\le -M;\\
	\bigl(f(x+\frac{1}{n})-\frac{1}{n}\bigr)\vee 0, & -M<x\le M;\\
	0, & x>M.
\end{cases}
\]
Then \(f_n\in\tilde{\mathcal{G}}\), and \(f_n\) is increasing with \(n\) (as \(f\) is non-increasing, both \(f(x+\frac{1}{n})\) and \(-\frac{1}{n}\) increase with \(n\)). Moreover, \(\bigvee_n f_n=f\), the join in \(\tilde{\mathcal{G}}\); indeed, pointwise \(\sup_n f_n(x)=f(x^+)\), and the upper-semicontinuous modification \(\widetilde{\sup_n f_n}\) recovers \(f\). Therefore, by \(\mathcal{E}^r_{p,r}\ll f\), there exists \(n_0\in\mathbb{N}\) such that \(\mathcal{E}^r_{p,r}\le f_{n_0}\). Hence \(\mathcal{E}^r_{p,r}(x)\le f_{n_0}(x)\) for all \(x\in[-M,M]\). Evaluating at \(r\), we have
\[
p=\mathcal{E}^r_{p,r}(r)\le f_{n_0}(r)=\bigl(f(r+\frac{1}{n_0})-\frac{1}{n_0}\bigr)\vee 0.
\]
As \(p>0\), this forces \(p\le f(r+\frac{1}{n_0})-\frac{1}{n_0}\). Since \(f\) is non-increasing, \(f(r+\frac{1}{n_0})\le f(r)\), whence
\[
p\le f(r)-\frac{1}{n_0}<f(r).
\]

$\Leftarrow$  Assume \(p<f(r)\) and \(r\) is a continuity point of \(f\). Suppose \(\mathcal{E}^r_{p,r}\not\ll f\). Then there is a directed family \(\{f_i\}_{i\in D}\) with \(f\le s:=\sup_{i\in D}f_i\) and \(\mathcal{E}^r_{p,r}\nleq f_i\) for all \(i\). From the shape of \(\mathcal{E}^r_{p,r}\) and the monotonicity of \(f_i\) we get \(f_i(r)<p\) for all \(i\). Since \(p<f(r)\) and \(f\) is continuous at \(r\), there is \(\delta>0\) with \(f(y)>p\) for all \(y\in(r-\delta,r+\delta)\cap[-M,M]\). Choose a continuity point \(y_0\) of \(s\) with \(r<y_0<r+\delta\). Then \(f(y_0)>p\); since \(f\le s\), \(s(y_0)\ge f(y_0)>p\). By Theorem~\ref{continuous point of super and infimum}, \(s(y_0)=\sup_{i\in D}f_i(y_0)\), so \(\sup_i f_i(y_0)>p\), and there is \(i_0\in D\) with \(f_{i_0}(y_0)>p\). As \(f_{i_0}\) is non-increasing and \(r<y_0\), \(f_{i_0}(r)\ge f_{i_0}(y_0)>p\), contradicting \(f_i(r)<p\) for all \(i\). Therefore \(\mathcal{E}^r_{p,r}\ll f\).

\((\tilde{\mathcal{G}},\le)\) is a complete lattice by Corollary \ref{complete lattice}.
For any \(f\in\tilde{\mathcal{G}}\), the preceding argument gives
\(f=\sup_{p,r}\mathcal{E}^r_{p,r}\) with each \(\mathcal{E}^r_{p,r}\ll f\). Since in a
complete lattice \(v_1,v_2\ll f\) implies \(v_1\vee v_2\ll f\), the set
\(\{v:v\ll f\}\) is directed; moreover
\(f=\sup_{p,r}\mathcal{E}^r_{p,r}\le\sup\{v:v\ll f\}\le f\), so \(f=\sup\{v:v\ll f\}\)
is a directed join. Therefore \((\tilde{\mathcal{G}},\le)\) is a continuous lattice.

\end{proof}

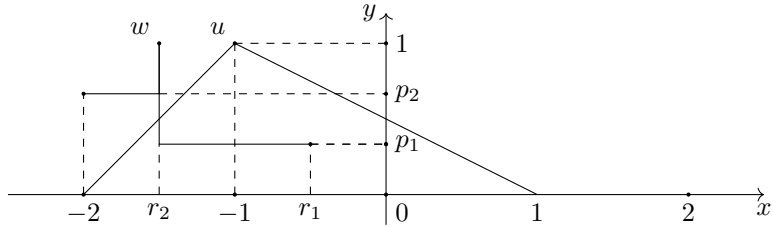
\begin{figure}[htbp]	
	\centering
	\begin{minipage}[t]{0.8\textwidth}
		\centering
		\begin{tikzpicture}	[scale=2]
			\draw[->] (-2.5,0)--(2.5,0);
			\draw[->] (0,-0.2)--(0,1.2); 
			\draw (-2,0)--(-1,1)--(1,0);
			\draw (-3/2,1)--(-3/2,1/3)--(-1/2,1/3); 
			\draw (-3/2,1)--(-3/2,2/3)--( -2,2/3);  
			
			\draw[dashed](0,1)--(-1,1);
			\draw[dashed](-1/2,1/3)--(0,1/3);
			\draw[dashed](-1,0)--(-1,1);
			\draw[dashed](-1/2,1/3)--(0,1/3);
			\draw[dashed](-3/2,0)--(-3/2,1/3);
			\draw[dashed](-1/2,0)--(-1/2,1/3);
			\draw[dashed](-2,0)--(-2,2/3);
			\draw[dashed](-3/2,2/3)--(0,2/3);
			
			\draw[fill] (0,1) circle(.01);
			\draw[fill] (0,0) circle(.01);
			\draw[fill] (-1,0) circle(.01);
			\draw[fill] (-1,1) circle(.01);
			\draw[fill] (0,1/3) circle(.01);
			\draw[fill] (0,1/3) circle(.01);
			\draw[fill] (0,2/3) circle(.01);
			\draw[fill] (-1/2,1/3) circle(.01);
			\draw[fill] (-2,2/3) circle(.01);
			
			\draw[fill] (-2,0) circle(.01);
			\draw[fill] (2,0) circle(.01);
			\draw[fill] (-3/2,1) circle(.01);
			
			\node (a) at (0,1) [right] {$1$};
			\node (b) at (-1,1) [above left] {$u$};
			\node (e) at (0,0) [below right] {$0$};
			\node (f) at  (-1,0) [below]{$-1$};
			\node (a) at (-2,0) [ below] {$-2$};
			\node (a) at (1,0) [ below] {$1$};
			\node (a) at (2,0) [ below] {$2$};
			\node (j) at (0,1/3) [ right] {$p_1$};
			\node (k) at (-1/2,0) [below] {$r_1$};
			\node (m) at (0,2/3) [ right] {$p_2$};
			\node (n) at (-3/2,0) [below] {$r_2$};
            \node (s) at (-3/2,1) [above left] {$w$};
       \node (x) at (2.5,0) [ below] {$x$};
        \node (y) at (0,1.2) [left] {$y$};
		\end{tikzpicture} 
	\end{minipage}
	\caption{ the graph of $u=(-2,-1,1)$, \(\mathcal{E}^r_{p_1,r_1}\) and   \(\mathcal{E}^\ell _{p_2,r_2} ,w=  \mathcal{E}^\ell _{p_2,r_2}\wedge \mathcal E^r_{p_1,r_1}\). }\label{crip order in function space}
\end{figure}

\begin{prop} \label{right branch way below}
	Let \( \& \) be a continuous t-norm satisfying the (S) condition.  For every \(f\in \tilde{\mathcal{G}}\), if \(r\) is a continuous point of \(f\) and \(p < f(r)\), 
	then \( \mathfrak{w}(\mathcal{E}^r_{p,r},f)=1\).  
\end{prop}
\begin{proof}

    Suppose that \(p< f(r) \).
	Let \(g\) be the Yoneda limit of a forward Cauchy net \(\{g_i\}_{i \in D}\) in \(\tilde{\mathcal{G}}\) with \(\phi=\sup_{i\in D} \inf_{j\ge i} P(-,g_j )\).
    \begin{align*}
		\phi (\mathcal{E}^r_{p,r})=& \sup_{i\in D} \inf_{j\ge i} P(\mathcal{E}^r_{p,r},g_j )\\
		=&\sup_{i\in D}P(\mathcal{E}^r_{p,r}, \inf_{j\ge i}g_j )\\
		=&\sup_{i\in D}\left\{\inf_{x\in [-M,M]} (\mathcal{E}^r_{p,r}(x)\rightarrow \inf_{j\ge i}g_j (x))\right\}\\
		=&\sup_{i\in D}( p\rightarrow \inf_{j\ge i}g_j (r)).
	\end{align*}
	By  Theorem \ref{continuous and except for countable point},  there exists \(r_n\searrow r\) such that \[P(f,g)\le \inf_{n\in \mathbb{N}} (f(r_n)\rightarrow g(r_n))\] with each \(r_n \)  being a continuous point of  \(g\).
	For  every \(n\in \mathbb{N}\), \(g(r_n)\le g(r)\) by the monotonicity of \(g\). 
       \begin{align*}
             \inf_{n\in \mathbb{N}} (f(r_n)\rightarrow g(r_n)) &\le
          \lim_n f(r_n) \rightarrow \lim_{n} g(r_n) \tag{ By adjoint of a continuous \(\&\)}\\
          &\le f(r)\rightarrow \lim_{n} \sup_{i\in D} \inf_{j\ge i}g_j (r_n). \tag{ Since \(r\) is a continuous point of \(f\) and Theorem \ref{continuous point of super and infimum}}\\
         &\le f(r)\rightarrow \sup_{i\in D} \inf_{j\ge i}g_j (r) \tag{Since \(g\) is monotonic}
     \end{align*}
		Since \(\&\) satisfies the (S) condition and \(p<f(r)\), it follows that
		\[ P(f,g)\le f(r)\rightarrow \sup_{i\in D} \inf_{j\ge i}g_j (r) \le \sup_{i\in D}( p\rightarrow \inf_{j\ge i}g_j (r)).\]
		Hence, \( P(f,g)\le \phi (\mathcal{E}^r_{p,r})\) and \( \mathfrak{w}(\mathcal{E}^r_{p,r},f)=1\).
\end{proof}
\begin{lem}\label{finite join of enriched way below}
Let \(X\) be a conically cocomplete \([0,1]\)-enriched category.
If \(\mathfrak w(x_k,x)=1\) for \(k=1,\ldots,n\), then
\[
\mathfrak w\left(\bigvee_{k=1}^n x_k,x\right)=1.
\]
\end{lem}

\begin{proof}
Every forward Cauchy ideal \(\phi\) preserves non-empty finite
conical joins in the form
\(\phi(\bigvee_{k=1}^n x_k)=\bigwedge_{k=1}^n\phi(x_k)\).
Indeed, this follows from the representation
\(\phi(y)=\sup_i\inf_{j\ge i}X(y,z_j)\), the conical-join identity
\(X(\bigvee_kx_k,z_j)=\bigwedge_kX(x_k,z_j)\), and the directedness
of the index set.

Now \(\mathfrak w(x_k,x)=1\) implies
\(X(x,\operatorname{colim}\phi)\le\phi(x_k)\) for every \(k\).
Hence
\[
X(x,\operatorname{colim}\phi)
\le\bigwedge_{k=1}^n\phi(x_k)
=\phi\left(\bigvee_{k=1}^n x_k\right).
\]
The conclusion follows directly from the definition of
\(\mathfrak w\).
\end{proof}
For a nonempty subset \(A\) of a lattice, define
\[
\operatorname{FinJoin}(A)
 :=\left\{\bigvee_{k=1}^{n}s_k
 \mid n\ge1,\ s_k\in A\right\}.
\]

\begin{prop}\label{fuzzy forward Cauchy ideal}
 Define
\[
D_f^r=
\operatorname{FinJoin}\left(
\left\{\mathcal E^r_{p,r}
\mid 0<p<f(r),\ r\text{ is a continuity point of }f\right\}
\right)
\]
where the joins are taken in
\((\widetilde{\mathcal G},\leq)\).
	\begin{align*}
		\phi_f(g)=\begin{cases}
			P(g,f), & f= \operatorname{lim} (\sup_{g\in\tilde{\mathcal{G} }} P(g,-)); \\
			\sup_{h\in D^r_f} P(g,h), & otherwise.
		\end{cases} 
	\end{align*}
	Then \( \phi_f \) is forward Cauchy ideal  and \(\operatorname{colim}\phi_f=f\).
\end{prop}
\begin{proof}
    The set \(D_f^r\) is directed. Indeed, if
\(h_1=\bigvee_{i=1}^{m}a_i\) and
\(h_2=\bigvee_{j=1}^{n}b_j\) belong to \(D_f^r\), then
\[
h_1\vee h_2
=
\bigvee\{a_1,\ldots,a_m,b_1,\ldots,b_n\}
\in D_f^r
\]
is a common upper bound of \(h_1\) and \(h_2\).
Moreover, every member of \(D_f^r\) is way below \(f\), since
finite joins of elements way below \(f\) are again way below \(f\).
Finally,
\[
\bigvee D_f^r=\bigvee\mathcal E^r_{p,r}=f.
\]
Consequently, \(\phi_f=\sup_{h\in D_f^r}P(-,h)\) is a forward
Cauchy ideal and \(\operatorname{colim}\phi_f=f\).
\end{proof}

\begin{thm}\label{G is a real enriched domain}
	Let \( \& \) be a continuous t-norm satisfying the (S) condition. \((\tilde{\mathcal{G} },P)\) is a \([0,1]\)-enriched domain, i.e.,
	\[
	\mathfrak{w}(g, f) = \inf_{\phi \in \mathcal{I}\tilde{\mathcal{G} }} \left( P(f, \operatorname{colim} \phi) \to \phi(g) \right)=\phi_f(g).
	\]
\end{thm}
\begin{proof}
	The verification is straightforward for the case \( f = \operatorname{lim} (\sup_{g \in \tilde{\mathcal{G}}} P(g,-)) \). Hence, we only need to consider the remaining case.
Because of Proposition \ref{right branch way below} and Lemma\ref{finite join of enriched way below}, we have \[
  \mathfrak{w} (h,f)=1, \text{ for every } h\in D^r_f. 
  \]
	For any \(\phi \in \mathcal{I}\tilde{\mathcal{G} }\), 
	\begin{align*}
		P(f,\operatorname{colim}\phi)\le \inf_{h\in D_f^r} \phi (h)
		\iff & P(f,\operatorname{colim}\phi)\le \inf_{h\in D_f^r}\inf_{g\in \tilde{\mathcal{G} }}\left(P(g,h) \rightarrow \phi (g)\right)\\
		\iff&P(f,\operatorname{colim}\phi)\le \inf_{g\in \tilde{\mathcal{G} }}\inf_{h\in D_f^r} \left(P(g,h) \rightarrow \phi (g)\right) \\
		\iff&P(f,\operatorname{colim}\phi)\le \left(\sup_{h\in D_f^r} P(g,h) \right)\rightarrow \phi (g) \tag{ for every \(g\in \tilde{\mathcal{G} }\)}\\
		\iff&\sup_{h\in D_f^r} P(g,h) \le P(f,\operatorname{colim}\phi) \rightarrow \phi (g) \tag{ for every \(g\in \tilde{\mathcal{G} }\)} \\
		\iff& \phi_f(g)\le \mathfrak{w}(g, f) \tag{ for every \(g\in \tilde{\mathcal{G} }\)}.
	\end{align*}
	The converse inequality is trivial. Hence,  \(\mathfrak{w}(g, f) =\phi_f(g) \). Therefore \((\tilde{\mathcal{G} },P)\) is a \([0,1]\)-enriched domain.
\end{proof} 
\begin{definition}
	Let  \(\&\) be a  left-continuous t-norm. For any \(f,g\in \tilde{\mathcal{F}}\),
 \[P(f,g)=\inf_{x\in [-M,M]} \left(g(x)\rightarrow f(x)\right).\]
\end{definition}
\begin{cor}
	Let  \(\&\) be a continuous t-norm. \((\tilde{\mathcal{F}},P)\) is a separated, conically cocomplete  and  conically complete \([0,1]\)-enriched category.
\end{cor}
\begin{cor}\label{crisp order in left branch and left branch way below}
	\((\tilde{\mathcal{F}})_0=(\tilde{\mathcal{F}},\ge )\) is a continuous lattice.  For any \(f\in \tilde{\mathcal{F} } \) and any \(0<p<1\), if \(r\) is a continuous point of \(f\), then \(\mathcal{E}^\ell_{p,r}\ll f\) if and only if \( p>f(r)\).
 \[\mathcal{E}^\ell_{p,r}=\begin{cases}
	0, & x<-M;\\
	p, & -M\le x< r;\\
	1, & r\le x.
\end{cases}\]
\end{cor}

\begin{cor}\label{left branch way below}
	Let \( \& \) be a continuous t-norm satisfying the (S) condition. For any  \(f\in \tilde{\mathcal{F}} \), if \(r\) is a continuous point of \(f\) and \( f(r)<p\), then \(\mathfrak{w}(\mathcal{E}^\ell_{p,r},f)=1\),  where \(\mathcal{E}^\ell_{p,r}\)
(denoted \(\mathcal{E}^\ell_{p_2,r_2}\) in Figure \ref{crip order in function space})
and \(f=u^\ell\).
\end{cor}
\begin{proof}

    Suppose that \(f(r)<p \).
	Let \(g\) be the Yoneda limit of a forward Cauchy net \(\{g_i\}_{i \in D}\) in \(\tilde{\mathcal{F}}\) with \(\phi =\sup_{i\in D} \inf_{j\ge i} P(-,g_j )\). \begin{align*}
		\phi (\mathcal{E}^\ell_{p,r})=& \sup_{i\in D} \inf_{j\ge i} P(\mathcal{E}^\ell_{p,r},g_j )\\
		=&\sup_{i\in D}P(\mathcal{E}^\ell_{p,r}, \sup_{j\ge i}g_j )\\
		=&\sup_{i\in D}\left\{\inf_{x\in [-M,M]} ( \sup_{j\ge i}g_j (x)\rightarrow \mathcal{E}^\ell_{p,r}(x) )\right\}\\
		=&\sup_{i\in D}( \sup_{j\ge i}g_j (r^-) \rightarrow p ).\tag{\(g_j(r^-)\) is the left limit of \(g_j\) at \(x=r\)}
	\end{align*}
	
	By Theorem \ref{Yoneda limit of P} and \(\inf_{i\in D}\sup_{j\ge i}g_j(r^-)\le g(r)\), it follows that \[P(f,g)\le  g(r) \rightarrow f(r)=\left(\inf_{i\in D} \widetilde{\sup_{j\ge i}g_j}(r)\right)\rightarrow f(r)\le \left(\inf_{i\in D}\sup_{j\ge i}g_j(r^-)\right)\rightarrow f(r).\]  

    \begin{enumerate}[label={\rm(\arabic*)}]
    \item Assume \(\displaystyle\inf_{i \in D} \sup_{j \ge i} g_j (r^-) > p\). Since  \(\&\) satisfies the (S) condition and \(f(r) < p\), it follows that
    \[
    \sup_{i \in D} \Bigl( \sup_{j \ge i} g_j (r^-) \to p \Bigl) 
    = \Bigl( \inf_{i \in D} \sup_{j \ge i} g_j (r^-) \Bigr) \to p 
    \ge \Bigl( \inf_{i \in D} \sup_{j \ge i} g_j(r^-) \Bigr) \to f(r).
    \]
    Consequently, \(P(f, g) \le \phi(\mathcal{E}^\ell_{p, r})\) and thus \(\mathfrak{w}(\mathcal{E}^\ell_{p, r}, f) = 1\).

    \item Assume \(\displaystyle\inf_{i \in D} \sup_{j \ge i} g_j (r^-) = p\). Using condition (S) for \(\&\) and the fact that \(f(r) < p\), we obtain
    \[
    \Bigl( \inf_{i \in D} \sup_{j \ge i} g_j(r^-) \Bigr) \to f(r)
    = \sup_{i \in D} \Bigl( \sup_{j \ge i} g_j (r^-) \to f(r) \Bigl)
    \le \sup_{i \in D} \Bigl( \sup_{j \ge i} g_j (r^-) \to p \Bigl).
    \]
    Hence, \(P(f, g) \le \phi(\mathcal{E}^\ell_{p, r})\) and therefore \(\mathfrak{w}(\mathcal{E}^\ell_{p, r}, f) = 1\).

    \item Assume \(\displaystyle\inf_{i \in D} \sup_{j \ge i} g_j (r^-) < p\). Then
    \[
    \sup_{i \in D} \Bigl( \sup_{j \ge i} g_j (r^-) \to p \Bigl) = 1,
    \]
    which immediately implies \(\mathfrak{w}(\mathcal{E}^\ell_{p, r}, f) = 1\).
\end{enumerate}
In summary, in all three cases we have shown that \(\mathfrak{w}(\mathcal{E}^\ell_{p, r}, f) = 1\).

\end{proof}

\begin{cor}
	We define:
    \[
   D_f^\ell=
\operatorname{FinJoin}\left(
\left\{\mathcal E^\ell_{p,r}
\mid f(r)<p<1,\ r\text{ is a continuity point of }f\right\}
\right).
    \]
	\begin{align*}
		\phi_f(g)=\begin{cases}
			P(g,f), & f= \operatorname{lim} (\sup_{g\in\tilde{\mathcal{F}}}P(g,-)); \\
			\sup_{h\in D^\ell_f} P(g,h), & otherwise.
		\end{cases} 
	\end{align*}
	Then \( \phi_f \) is a forward Cauchy ideal and \(\operatorname{colim}\phi_f=f\).
\end{cor}
\begin{proof}
By an argument similar to that of  Corollary \ref{fuzzy forward Cauchy ideal}.
\end{proof}
\begin{thm}\label{F is a real enriched domain}
	Let \( \& \) be a continuous t-norm satisfying the (S) condition. \((\tilde{\mathcal{F}},P)\) is a \([0,1]\)-enriched domain, i.e., 
	\[
	\mathfrak{w}(g, f) = \inf_{\phi \in \mathcal{I}\tilde{\mathcal{F}}} \left( P(f, \operatorname{colim} \phi) \to \phi(g) \right)=\phi_f(g).
	\]
\end{thm}
\begin{proof}
By an argument similar to that of Theorem \ref{G is a real enriched domain}.
\end{proof}

\begin{prop}\label{underlying order is continuous lattice}
Let $\&$ be a continuous t-norm. Then the underlying ordered set
$(\mathbb{E}^1_{[-M,M]})_0=(\mathbb{E}^1_{[-M,M]},\sqsubseteq)$ is a continuous
lattice, where
\[
u\sqsubseteq v\ \Longleftrightarrow\ P(u,v)=1
\ \Longleftrightarrow\ v^\ell\le u^\ell\ \ \text{and}\ \ u^r\le v^r.
\]
\end{prop}
\begin{proof}
Recall that $(\mathbb{E}^1_{[-M,M]})_0$ is a complete lattice, and for every
$A\subseteq\mathbb{E}^1_{[-M,M]}$ the join satisfies (Theorem
\ref{continuous point of super and infimum})
\[
\Big(\bigvee A\Big)^\ell=\inf_{u\in A}u^\ell,\qquad
\Big(\bigvee A\Big)^r=\widetilde{\,\sup_{u\in A}u^r\,}.
\]
For \(u\in\mathbb E^1_{[-M,M]}\), define \(D_u^0\) as follows.

If \(u_1^->-M\), let
\[
D_u^0=
\left\{
w:
\begin{array}{l}
w^\ell=\mathcal E^\ell_{p_2,r_2},\quad
w^r=\mathcal E^r_{p_1,r_1}, \mathcal E^r_{p_1,r_1}\in \tilde{\mathcal{G}}_{[r_2,M]},\\
u^\ell(r_2)<p_2<1,\quad
0< p_1<u^r(r_1),\\
r_2<r_1,
\end{array}
\right\},
\]
where \(\tilde{\mathcal{G}}_{[r_2,M]}\) denotes the restriction of \(\tilde{\mathcal{G}}\) to the interval \([r_2,M]\) and  \(r_2,r_1\) are continuity points of \(u^\ell,u^r\),
respectively.

If \(u_1^-=-M\), then \(u^\ell=\mathcal E^\ell_{1,-M}\), and we
define \(D_u^0\) by the right-branch condition only:
\[
D_u^0=
\left\{
w:
w^\ell=\mathcal E^\ell_{1,-M},\quad
w^r=\mathcal E^r_{p_1,r_1},\quad
0< p_1<u^r(r_1)
\right\}.
\]

If \(u_1^-=-M\), then all the left branches occurring in
\(D_u^0\) are equal to \(\mathcal E^\ell_{1,-M}\). Hence the
way-below relation and the approximation of \(u\) reduce to the
right-branch case established in Proposition
\ref{way-below of crisp order}. We may therefore assume in the
sequel that \(u_1^->-M\). Fix $w\in D_u^0$, and let $A=\{u_i\}_{i\in D}$ be directed with $u\sqsubseteq e:=\bigvee A$.
Then
\[
e^\ell=\inf_{i\in D}u_i^\ell\le u^\ell,\qquad u^r\le e^r=\widetilde{\,\sup_{i\in D}u_i^r\,}.
\]
For the left branch, since $u^\ell(r_2)<p_2$ and $e^\ell=\inf_i u_i^\ell\le u^\ell$,
\[
\inf_{i\in D}u_i^\ell(r_2)=e^\ell(r_2)\le u^\ell(r_2)<p_2,
\]
so there is $i_\ell\in D$ with $u_{i_\ell}^\ell(r_2)<p_2$. As $u_{i_\ell}^\ell$ is non-decreasing,
\[
x<r_2\ \Rightarrow\ u_{i_\ell}^\ell(x)\le u_{i_\ell}^\ell(r_2)<p_2=w^\ell(x),
\]
while $w^\ell(x)=1\ge u_{i_\ell}^\ell(x)$ for $x\ge r_2$ and $w^\ell(x)=0=u_{i_\ell}^\ell(x)$
for $x<-M$; hence
\begin{equation}\label{eq:L}
u_{i_\ell}^\ell\le w^\ell.
\end{equation}
For the right branch, the index \(i_\ell\) chosen above also
provides the required value on \((-\infty,r_2]\). Indeed,
\(u_{i_\ell}^\ell(r_2)<p_2<1\) implies \(u_{i_\ell}^\ell(r_2)<1\).
Since \(u_{i_\ell}^\ell(r_2)\vee u_{i_\ell}^r(r_2)=1\), we have
\(u_{i_\ell}^r(r_2)=1\). By the monotonicity of \(u_{i_\ell}^r\),
\begin{equation}\label{eq:R1}
u_{i_\ell}^r(x)=1\qquad (x\le r_2).
\end{equation}

Since \(u^r\) is continuous at \(r_1\) and \(p_1<u^r(r_1)\),
choose \(x^*>r_1\), sufficiently close to \(r_1\), such that
\(u^r(x^*)>p_1\) and \(x^*\) is a continuity point of \(e^r\).
Then \(e^r(x^*)=\sup_{i\in D}u_i^r(x^*)\ge u^r(x^*)>p_1\).
Hence there exists \(i_r'\in D\) such that \(u_{i_r'}^r(x^*)>p_1\).
Since \(u_{i_r'}^r\) is non-increasing,
\begin{equation}\label{eq:R2}
u_{i_r'}^r(x)>p_1=w^r(x)
\qquad (r_2<x\le r_1).
\end{equation}

Choose \(i_r\in D\) with \(i_r\ge i_\ell,i_r'\). From
\eqref{eq:R1}, \eqref{eq:R2}, and \(w^r(x)=0\) for \(x>r_1\),
we obtain \(w^r\le u_{i_r}^r\). Moreover, since \(i_r\ge i_\ell\),
equation \eqref{eq:L} gives \(u_{i_r}^\ell\le u_{i_\ell}^\ell\le w^\ell\).
Thus \(w\sqsubseteq u_{i_r}\), and consequently \(w\ll_0u\).

It remains to prove that \(\bigvee D_u^0=u\). The case
\(u=\chi_{-M}\) is immediate. Assume \(u\ne\chi_{-M}\).
By the definition of \(D_u^0\), every \(w\in D_u^0\) satisfies
\(
u^\ell\le w^\ell, w^r\le u^r,
\)
and hence \(w\sqsubseteq u\).
Conversely, by varying \(r_2,r_1\) over the continuity points of
\(u^\ell,u^r\), respectively, and letting
\(p_2\searrow u^\ell(r_2)\) and
\(p_1\nearrow u^r(r_1)\), one directly verifies that
\[
\inf_{w\in D_u^0}w^\ell=u^\ell,
\qquad
\widetilde{\sup_{w\in D_u^0}w^r}=u^r.
\]
Here the second equality uses the density of the continuity points
of \(u^r\) and the left-continuity of \(u^r\). If \(u_1^-=-M\),
the first equality is automatic and the same right-branch argument
applies. Therefore, by the formula for joins in
\((\mathbb E^1_{[-M,M]})_0\),
\[
\left(\bigvee D_u^0\right)^\ell=u^\ell,
\qquad
\left(\bigvee D_u^0\right)^r=u^r.
\]
Hence \(\bigvee D_u^0=u\).

Finally, the set $\twoheaddownarrow u:=\{v:v\ll_0 u\}$ is directed, since in a complete lattice
$v_1,v_2\ll_0 u$ implies $v_1\vee v_2\ll_0 u$. As $D_u^0\subseteq\twoheaddownarrow u$ and
$u=\bigvee D_u^0$, we have
\[
u=\bigvee D_u^0\ \sqsubseteq\ \bigvee\twoheaddownarrow u\ \sqsubseteq\ u,
\]
i.e. $u=\bigvee\twoheaddownarrow u$ with $\twoheaddownarrow u$ directed. Since $u$ is arbitrary, every
element is the directed supremum of elements way below it, so $(\mathbb{E}^1_{[-M,M]})_0$
is a continuous lattice.
\end{proof}
For \(u\ne\chi_{-M}\), put
\[
D_u=\operatorname{FinJoin}(D_u^0)
=
\left\{
\bigvee_{k=1}^n b_k
\ \middle|\
n\ge1,\ b_k\in D_u^0
\right\},
\]
where the joins are taken in
\((\mathbb E^1_{[-M,M]})_0\).

\begin{cor}
	Let 
	\begin{align*}
		\Phi_u(v)=\begin{cases}
			P(v,u), & u= \operatorname{lim} (\sup_{e\in\mathbb{E}^1_{[-M,M]}} P(e,-) ); \\
			\sup_{w\in D_u}  P(v,w), & otherwise.
		\end{cases} 
	\end{align*}
	Then \(\Phi_u\) is a forward Cauchy ideal and  \(\operatorname{colim}\Phi_u=u \).
\end{cor}
\begin{proof}
If \(u=\chi_{-M}\), then \(\Phi_u=P(-,u)\) is generated by the
constant net at \(u\), and its colimit is \(u\).

Assume that \(u\ne\chi_{-M}\). By Proposition
\ref{underlying order is continuous lattice},
\(D_u^0\subseteq\{w:w\ll_0u\}\) and \(\bigvee D_u^0=u\).
The set \(D_u=\operatorname{FinJoin}(D_u^0)\) is directed,
because the join of two of its members is again a finite join
of elements of \(D_u^0\). Moreover, \(\bigvee D_u=u\).
Regarded as a net indexed by its inherited order, \(D_u\) is an
increasing, and hence forward Cauchy, net. Its associated
forward Cauchy ideal is \(\sup_{w\in D_u}P(-,w)=\Phi_u\).
Since the colimit of this conical weight is \(\bigvee D_u\), it
follows that \(\operatorname{colim}\Phi_u=\bigvee D_u=u\).
\end{proof}

\begin{thm} \label{Theorem (S) condition imply a real enriched domain}
	Let \( \& \) be a continuous t-norm satisfying the  (S) condition. \((\mathbb{E}^1_{[-M,M]},P)\) is a \([0,1]\)-enriched domain, i.e.,
	\[
	\mathfrak{w}(v, u) = \inf_{\Phi \in \mathcal{I}\mathbb{E}^1_{[-M,M]}} \left( P(u, \operatorname{colim} \Phi) \to \Phi(v) \right)=\Phi_u(v).
	\]
\end{thm}
\begin{proof}
Let
\[
u_*=\operatorname{lim}\Bigl(\sup_{e\in\mathbb E^1_{[-M,M]}}P(e,-)\Bigr).
\]
By the definition of a limit, \(P(x,u_*)=\inf_{e}P(x,e)\) for every \(x\);
in particular \(P(u_*,e)=1\) for every \(e\), so \(u_*\) is the least
element of \((\mathbb E^1_{[-M,M]})_0\).

\smallskip
\noindent\textbf{Case 1.} \(u=u_*\). Then \(\Phi_u=P(-,u)\). For any
\(\Phi\in\mathcal I\mathbb E^1_{[-M,M]}\), we have
\[
\Phi(u)=1,\qquad
P(y,u)\le\Phi(u)\to\Phi(y)=\Phi(y).
\]
Since \(P(u,\operatorname{colim}\Phi)=1\), it follows that
\(P(y,u)\le P(u,\operatorname{colim}\Phi)\to\Phi(y)\). Taking the infimum
over \(\Phi\) yields \(P(y,u)\le\mathfrak w(y,u)\), and the reverse
inequality follows on taking \(\Phi=P(-,u)\). Hence
\(\mathfrak w(-,u)=\Phi_u\).

\smallskip
\noindent\textbf{Case 2.} \(u\ne u_*\). Fix \(y\in\mathbb E^1_{[-M,M]}\), and
let \(\Phi\in\mathcal I\mathbb E^1_{[-M,M]}\) be generated by a forward
Cauchy net \(\{v_i\}_{i\in I}\), so that
\[
\Phi(a)=\sup_{i\in I}\inf_{j\ge i}P(a,v_j)
\qquad\text{and}\qquad
v=\operatorname{colim}\Phi.
\]
Set
\[
L=\inf_{x\in[-M,M]}\bigl(v^\ell(x)\to u^\ell(x)\bigr),\qquad
R=\inf_{x\in[-M,M]}\bigl(u^r(x)\to v^r(x)\bigr),
\]
so that \(P(u,v)=L\wedge R\). We claim that
\begin{equation}\label{eq:PlePhi}
P(u,v)\le\Phi(b)\qquad\text{for every }b\in D_u^0.
\end{equation}

Suppose first that \(u_1^->-M\). Write \(b^\ell=\mathcal E^\ell_{p_2,r_2}\)
and
\[
b^r(x)=\begin{cases}
1,&x\le r_2,\\
p_1,&r_2<x\le r_1,\\
0,&x>r_1,
\end{cases}
\]
where \(u^\ell(r_2)<p_2<1\), \(0<p_1<u^r(r_1)\), and \(r_2,r_1\) are
continuity points of \(u^\ell,u^r\), respectively. A direct computation
gives \(\Phi(b)=A_\ell\wedge A_r\wedge A_0\), where
\[
A_\ell=\sup_{i\in I}\Bigl(\sup_{j\ge i}v_j^\ell(r_2^-)\to p_2\Bigr),\qquad
A_r=\sup_{i\in I}\Bigl(p_1\to\inf_{j\ge i}v_j^r(r_1)\Bigr),\qquad
A_0=\sup_{i\in I}\inf_{j\ge i}v_j^r(r_2).
\]
Since \(L=P(u^\ell,v^\ell)\) and \(R=P(u^r,v^r)\), Corollary
\ref{left branch way below} gives \(L\le A_\ell\), while applying Proposition \ref{right branch way below} to the restrictions
of \(u^r\) and \(v_j^r\) to \([r_2,M]\), we obtain \(R\le A_r\).

It remains to verify \(R\le A_0\). Since \(u^\ell(r_2)<p_2<1\), we have
\(r_2<u_1^-\). Let \(B(x)=\sup_{i\in I}\inf_{j\ge i}v_j^r(x)\), and choose a
common continuity point \(t\) of \(u^r\) and \(v^r\) with \(r_2<t<u_1^-\).
By the right-branch formula for the colimit, \(v^r(t)=B(t)\); moreover
\(u^r(t)=1\), and \(B\) is non-increasing. Therefore
\[
R\le u^r(t)\to v^r(t)=v^r(t)=B(t)\le B(r_2)=A_0.
\]
Together with \(L\le A_\ell\) and \(R\le A_r\), this proves
\eqref{eq:PlePhi} in the case \(u_1^->-M\).

Suppose now that \(u_1^-=-M\). Then \(u^\ell=b^\ell=\mathcal E^\ell_{1,-M}\),
so the left-branch term \(L\) equals \(1\) and \(P(u,v)=R\). If \(r_1>-M\),
Proposition \ref{right branch way below} yields
\[
P(u,v)=R\le\sup_{i\in I}\Bigl(p_1\to\inf_{j\ge i}v_j^r(r_1)\Bigr)=\Phi(b).
\]
If \(r_1=-M\), then \(b=\chi_{-M}=u_*\), whence \(\Phi(b)=1\), and the
inequality is trivial. Thus \eqref{eq:PlePhi} holds in all cases.

\smallskip
Since \(b\in D_u^0\) and \(\Phi\in\mathcal I\mathbb E^1_{[-M,M]}\) are arbitrary, the preceding argument yields
\(P(u,\operatorname{colim}\Phi)\le\Phi(b)\) for every \(b\in D_u^0\) and every \(\Phi\in\mathcal I\mathbb E^1_{[-M,M]}\).
Hence, by the definition of the way-below distributor, \(\mathfrak w(b,u)=1\) for every \(b\in D_u^0\).

Now let \(h\in D_u=\operatorname{FinJoin}(D_u^0)\), so that \(h=\bigvee_{k=1}^n b_k\) for some \(b_1,\ldots,b_n\in D_u^0\).
Since \(\mathfrak w(b_k,u)=1\) for each \(k\), Lemma~\ref{finite join of enriched way below} gives \(\mathfrak w(h,u)=1\).
Consequently, for every \(\Phi\in\mathcal I\mathbb E^1_{[-M,M]}\),
\[
1=\mathfrak w(h,u)\le P(u,\operatorname{colim}\Phi)\to\Phi(h),
\]
so \(P(u,\operatorname{colim}\Phi)\le\Phi(h)\). As \(h\in D_u\) is arbitrary, we obtain
\[
P(u,\operatorname{colim}\Phi)\le\inf_{h\in D_u}\Phi(h).
\]

For every \(h\in D_u\), the weight condition gives
\(P(y,h)\mathbin{\&}\Phi(h)\le\Phi(y)\). Combining this with
\(P(u,\operatorname{colim}\Phi)\le\Phi(h)\) yields
\(P(y,h)\mathbin{\&}P(u,\operatorname{colim}\Phi)\le\Phi(y)\), i.e.
\(P(y,h)\le P(u,\operatorname{colim}\Phi)\to\Phi(y)\). Taking the supremum
over \(h\in D_u\) and then the infimum over \(\Phi\), we obtain
\[
\Phi_u(y)=\sup_{h\in D_u}P(y,h)\le\mathfrak w(y,u).
\]

Conversely, \(\Phi_u\) is a forward Cauchy ideal with
\(\operatorname{colim}\Phi_u=u\). Taking \(\Phi=\Phi_u\) in the definition
of \(\mathfrak w\) gives
\[
\mathfrak w(y,u)\le P(u,u)\to\Phi_u(y)=\Phi_u(y).
\]
Hence \(\mathfrak w(-,u)=\Phi_u\). Since \(\Phi_u\) is a forward Cauchy
ideal with colimit \(u\), Proposition \ref{way-below} shows that
\((\mathbb E^1_{[-M,M]},P)\) is continuous. Being Yoneda complete by the
preceding result, it is a \([0,1]\)-enriched domain.
\end{proof}

A natural question arises: Is  the (S) condition necessary  for \(\mathbb{E}^1_{[-M,M]}\) to be a \([0,1]\)-enriched domain? To answer this, we first introduce several propositions and theorems.

Let \(X\) be a \([0,1]\)-enriched category. For each ideal \(I\) of the ordered set \(X_0\), the conical weight
\[
\Lambda(I) := \sup_{x \in I} X(-, x)
\]
is clearly a forward Cauchy ideal of \(X\). Conversely, if \(X\) is  conically  cocomplete and  conically complete; one verifies that for each forward Cauchy ideal \(\phi\) of \(X\), the set
\[
\chi(\phi) := \{ x \in X \mid \phi(x) = 1 \}
\]
is an ideal of the ordered set \(X_0\). 
\begin{thm}[\cite{Zhang2024IntroductoryNO}, Theorem 4.5]\label{adjoint relation between fuzzy order and crisp order}
    For each pair of maps \(f: X \to Y\) and \(g: Y \to X\) between \([0,1]\)-enriched categories, the following are equivalent:
\begin{enumerate}[label={\rm(\arabic*)}]
    \item Both \(f\) and \(g\) are functors, and \(f\) is left adjoint to \(g\).
    \item Both \(f\) and \(g\) are functors, and \(f: X_0 \to Y_0\) is left adjoint to \(g: Y_0 \to X_0\).
    \item For all \(x \in X\) and \(y \in Y\), \(Y(f(x), y) = X(x, g(y))\).
\end{enumerate}
\end{thm}

\begin{lem}[\cite{Zhang2024IntroductoryNO}, Lemma 17.19]
    If \(X\) is a  conically  cocomplete and  conically complete \([0,1]\)-enriched category and  \(\mathrm{Idl} X_0\) is the set of ideals of the ordered set \(X_0\), then \(\Lambda: \mathrm{Idl} X_0 \to (IX)_0\) is left adjoint to \(\chi: (IX)_0 \to \mathrm{Idl} X_0\).
\end{lem}

\begin{prop}\label{fuzzy ideal in crisp order}
   Let \(X\) be a separated, conically  cocomplete and  conically complete \([0,1]\)-enriched category. Then for each weight \(\phi\) of \(X\), the following are equivalent:
\begin{enumerate} [label={\rm(\arabic*)}]
    \item \(\phi\) is a forward Cauchy ideal of \(X\).
    \item \(\phi = \Lambda(I)\) for some ideal \(I\) of the complete lattice \(X_0\).
\end{enumerate}
In this case,   \(\operatorname{colim} \phi = \sup (\chi(\phi))\) in \(X_0\).
\end{prop}
\begin{proof}
    $(1) \implies (2)$ We show that the ideal $\chi(\phi)$ satisfies the requirement; that is, $\Lambda \circ \chi(\phi) = \phi$. Since $\Lambda$ is left adjoint to $\chi$, we have $\Lambda \circ \chi(\phi) \leq \phi$. For the converse inequality, pick a forward Cauchy net $\{x_i\}_{i \in D}$ of $X$ such that
\[
\phi(x) = \sup_{i \in D} \inf_{j \geq i} X(x, x_j).
\]
Let
\[
I = \left\{ x \in X \mid x \sqsubseteq \inf_{j \geq i} x_j \text{ for some } i \in D \right\},
\]
where $\inf_{j \geq i} x_j$ denotes the meet in $X_0$. Then, $I$ is an ideal of the ordered set $X_0$.
\[
\phi = \sup_{i \in D} \inf_{j \geq i} X(-, x_j)= \sup_{x \in I} X(-, x).
\]
so $I \subseteq \chi(\phi)$ and consequently, $\phi \leq \Lambda \circ \chi(\phi)$.

$(2) \implies (1)$ Trivial, since every ideal of the ordered set $X_0$ may be viewed as a forward Cauchy net of $X$.

Finally, we check that the colimit of \(\phi\) is the join  of \(\chi(\phi)\) in \(X_0\). Let \(I\) be the ideal of \(X_0\) given in $(1)\implies (2)$. 
\[
\operatorname{colim }\phi = \operatorname{colim }\sup_{x \in I} X(-, x)=  \sup I \le   \sup \chi(\phi).
\]
Conversely, \(\Lambda \circ \chi(\phi) \leq \phi\), so \(\operatorname{colim }\Lambda\circ \chi(\phi) =\sup \chi(\phi)\le \operatorname{colim }\phi \).
\end{proof}
\begin{thm}\label{denote way below in continuous domain}
    For each separated,  conically  cocomplete and  conically complete  \([0,1]\)-enriched category $X$, and \(X_0\) is a continuous lattice. The following are equivalent:
\begin{enumerate} [label={\rm(\arabic*)}]
    \item $X$ is a \([0,1]\)-enriched domain.
    \item For each $x \in X$ and each forward Cauchy ideal $\phi$ of $X$,
    \[
    X(x, \mathrm{colim}\,\phi) = \inf_{y \ll x} \phi(y),
    \]
    where $\ll$ denotes the way below relation in $X_0$.
    \item The map
    \[
    \mathrm{d}: X \to \mathcal{I}X,\quad \mathrm{d}(x) = \sup_{y \ll x} X(-, y)
    \]
    is a functor, where $\ll$ denotes the way below relation in $X_0$.
\end{enumerate}
\end{thm}

\begin{proof}
$(1) \implies(2)$  To show that
\[
X(x, \operatorname{colim}\,\phi) = \inf_{y \ll x} \phi(y)
\]
for all $x \in X$ and all $\phi \in \mathcal{I}X$, it suffices to show that the left adjoint $\twoheaddownarrow: X \to \mathcal{I}X$ of $\mathrm{colim}: \mathcal{I}X \to X$ is given by $\twoheaddownarrow x = \sup_{y \ll x} \mathsf{y}(y)$, because
\[
\mathcal{I}X\left( \sup_{y \ll x} \mathsf{y}(y), \phi \right) = \mathrm{sub}_X\left( \sup_{y \ll x} \mathsf{y}(y), \phi \right) = \inf_{y \ll x} \phi(y).
\]

Since $x$ is a colimit of the forward Cauchy ideal $\sup_{y \ll x} \mathsf{y}(y)$, then $\twoheaddownarrow x \leq \sup_{y \ll x} \mathsf{y}(y)$. Since $\sup \chi(\twoheaddownarrow x) =  \mathrm{colim}\,\twoheaddownarrow x = x$ by Proposition \ref{fuzzy ideal in crisp order}, then every element way below $x$ in $X_0$ belongs to $\chi(\twoheaddownarrow x)$, so $\sup_{y \ll x} \mathsf{y}(y) \leq \twoheaddownarrow x$.

$(2)\implies(3)$ Obvious.

$(3)\implies(1)$ By virtue of Theorem \ref{adjoint relation between fuzzy order and crisp order} we only need to show that $\mathrm{d}: X_0 \to (\mathcal{I}X)_0$ is left adjoint to $\mathrm{colim}: (\mathcal{I}X)_0 \to X_0$. This follows directly from Proposition \ref{fuzzy ideal in crisp order}, which guarantees that for each $x \in X$, $\sup_{y \ll x} X(-, y)$ is the smallest ideal of $X$ (w.r.t. underlying order) that has $x$ as a colimit.
\end{proof}

\begin{thm}
  Let \( \& \) be a continuous t-norm. \((\mathbb{E}^1_{[-M,M]},P)\) is a \([0,1]\)-enriched domain if and  only if it satisfies  the (S) condition.  
\end{thm}

\begin{proof}[Proof of necessity]
By Proposition~\ref{conically cocomplete and conically complete in E}, 
$(\mathbb{E}^{1}_{[-M,M]},P)$ is separated, conically cocomplete and conically complete, 
and by Proposition~\ref{underlying order is continuous lattice} its underlying 
ordered set is a continuous lattice.  In view of Theorem~\ref{denote way below in continuous domain}, 
it suffices to prove that if
\[
\mathrm{d}: \mathbb{E}^{1}_{[-M,M]} \to \mathcal{I} \mathbb{E}^{1}_{[-M,M]},\quad
\mathrm{d}(u)=\sup_{w\ll u} P(-,w)
\]
is a functor, then $\&$ satisfies the (S) condition.

Suppose $\&$ fails the (S) condition. By Proposition~\ref{implication operator is continuous at every point}, 
there exist idempotents $p,q>0$ with $p<q$ such that $\&$ restricted to $[p,q]$ is isomorphic 
to the {\L}ukasiewicz t-norm.  Pick $x\in(p,q)$ and fix $M=1$. Define
\[
u(t)=\begin{cases}0,&t<-1,\\1,&t=-1,\\x,&-1<t\le1,\\0,&t>1,\end{cases}
\qquad
v(t)=\begin{cases}0,&t<-1,\\1,&t=-1,\\p,&-1<t\le1,\\0,&t>1.\end{cases}
\]
Then $u,v\in\mathbb E^1_{[-1,1]}$, and we have
\[
P(u,v)=\inf_t(u^r(t)\to v^r(t))=x\to p>p. \tag{$\ast$}
\]

\emph{Step 1.} Fix $t_0\in(-1,1)$ and $p<\alpha<x$. Let
\[
h(t)=\begin{cases}0,&t<-1\\1,&t=-1,\\\alpha,&-1<t\le t_0,\\0,&t>t_0.\end{cases}
\]  
By  Proposition
\ref{underlying order is continuous lattice},  \(h\in D_u^0\) implies \(h\ll_0u\).
Therefore
\[
d(u)(h)=\sup_{z\ll_0 u}P(h,z)\ge P(h,h)=1,
\]
that is $d(u)(h)=1$.

\emph{Step 2.} Take $p_n\nearrow p$ with $p_n<p$, and define
\[
v_n(t)=\begin{cases}0,&t<-1 \\1,&t=-1,\\p_n,&-1<t\le1,\\0,&t>1,\end{cases}
\]
Then $v_1\le_0 v_2\le_0\cdots$ and $\sup_0 v_n=v$. For any $w\ll_0 v$, the way-below
property yields $n$ with $w\le_0 v_n$, hence
\[
w^r(t_0)\le v^r_n(t_0)=p_n<p.
\]
Let $c=w^r(t_0)<p$. Since $h(t_0)=\alpha\in(p,q)$ while $c<p$, 
$\alpha\to c=c$. Thus
\[
P(h,w)\le h^r(t_0)\to w^r(t_0)=\alpha\to c=c<p,
\]
and taking the supremum over $w\ll_0 v$,
\[
d(v)(h)=\sup_{w\ll_0 v}P(h,w)\le p.
\]

\emph{Step 3.} Since \(d\) is assumed to be a functor, 
\[
P(u,v)\le d(u)(h)\to d(v)(h)=1\to d(v)(h)=d(v)(h)\le p,
\]
which contradicts $P(u,v)>p$. Hence $d$ is not a functor,
contradicting the domain assumption. Therefore $\&$ satisfies the (S) condition.
\end{proof}

\section{Conclusion}
This paper introduces a novel fuzzy order \(P\) on the space of fuzzy numbers \(\mathbb{E}^1\), grounded in the theory of \([0,1]\)-enriched categories. The construction is based on decomposing each fuzzy number into its left and right monotone upper-semicontinuous branches, and the order is defined by an inclusion-type formula involving the residual implication of a chosen left-continuous t-norm. We show that this order not only generalizes the natural orders on real numbers and interval numbers, but also captures the distributional and structural features of fuzzy numbers. A distinctive computational property is its independence from points of discontinuity (Theorem \ref{continuous and except for countable point}), which ensures robustness in practical evaluation.

Our completeness analysis reveals that the uniformly bounded subspace \(\mathbb{E}^1_{[-M,M]}\) possesses conical completeness, conical cocompleteness, and Yoneda completeness---structures that mirror the order-theoretic completeness of bounded real numbers. We further determine when this space is complete: \((\mathbb{E}^1_{[-M,M]},P)\) is complete---equivalently cocomplete, tensored, or cotensored---if and only if the t-norm is the minimum, \(\&=\wedge\) (Theorem~\ref{completeness iff Godel}). Thus, although conical and Yoneda completeness hold for every continuous t-norm, the stronger properties of completeness, cocompleteness, tensoredness and cotensoredness are exceptional and single out the Gödel t-norm. In contrast, the unrestricted space \(\mathbb{E}^1\) is neither finitely conically cocomplete nor finitely conically complete, and admits no nontrivial tensors or cotensors, highlighting the necessity of bounded-support constraints for a well-behaved order theory.

The main result of this paper is the characterization of continuity: under a continuous t-norm, \((\mathbb{E}^1_{[-M,M]}, P)\) constitutes a \([0,1]\)-enriched domain if and only if the t-norm satisfies the (S) condition. This finding establishes a precise boundary between t-norms that yield a continuous domain structure and those that do not, thereby providing a criterion for selecting appropriate t-norms in applications requiring approximation and fixed-point properties.

Several open problems remain for future investigation.

\begin{enumerate}[label=(\arabic*)]
    \item The present work focuses on bounded-support fuzzy numbers;
    extending the completeness and continuity results to the entire space
    $\mathbb{E}^1$---or to other important subclasses such as fuzzy numbers
    with unbounded support---would be a natural next step.

 \item From an applied perspective, the fuzzy order $P$ can serve as a
fuzzy relation for ranking in multi-criteria decision-making. By
thresholding $P(u,v) \geq \alpha$, one obtains a crisp partial order
sensitive to both the location and the shape of fuzzy numbers, with
$\alpha=1$ recovering the natural order on real and interval data. The
dependence of $P$ on the choice of a left-continuous t-norm further allows
the decision-maker to calibrate the strictness of the comparison---as
illustrated in Example~3.8. A systematic study of how the choice of
t-norm affects ranking outcomes in concrete decision-making problems,
together with a computational comparison between $P$ and established
ranking approaches, is left for future work.

 \item The relationship between the fuzzy order $P$ and the ranking methods surveyed in Section~1 merits further investigation. Unlike ranking index methods, $P$ is a fuzzy order that encodes a graded truth degree of inclusion, whereas ranking indices assign a crisp value to each fuzzy number and induce a total preorder. Furthermore, $P$ differs from the fuzzy relations used in the third class of approaches by Wang and Kerre~\cite{WANG2001387}: $\&$-transitivity involves numerical composition via a t-norm, whereas acyclicity concerns only the ordinal pattern $P(a,b) > P(b,a)$. For the minimum t-norm, it is known that $\&$-transitivity implies acyclicity~\cite{WANG2001387}; for other t-norms, however, this question remains open. A systematic computational comparison between $P$ and established ranking approaches is left for future work.

\end{enumerate}

In summary, this work provides a unified, category-theoretic foundation for ordering fuzzy numbers that is both mathematically rigorous and computationally tractable, opening new avenues for research in fuzzy optimization, fuzzy decision analysis, and the theory of \([0,1]\)-enriched domains.


\bibliographystyle{abbrv}

\end{document}